\documentclass[11pt,reqno]{amsart}
\usepackage{graphicx} 
\usepackage[dvipsnames]{xcolor}
\usepackage{amsmath}
\usepackage{amsfonts}
\usepackage{amssymb}
\usepackage{ytableau}
\usepackage{amsthm}
\usepackage{tikz}
\usetikzlibrary{tikzmark,decorations.pathreplacing}
\usepackage{hyperref}
\usepackage{color}
\usepackage{graphicx}
\usepackage{float}                
\usepackage{xcolor}               
\usetikzlibrary{decorations.pathreplacing}
\usetikzlibrary{patterns,patterns.meta}
\usepackage{graphicx}
\usepackage{tikz-cd}
\usepackage[numbers]{natbib}
\usepackage{appendix}
\usetikzlibrary{tqft}
\usetikzlibrary{decorations.pathreplacing} 
\usepackage{mathtools}
\usepackage{enumerate}
\usepackage{multicol}
\usepackage{seqsplit}
\usepackage{enumitem}
\usepackage{enumitem}

\usepackage{xy}

\input xy
\xyoption{all}

\theoremstyle{plain}
\newtheorem{theorem}{Theorem}[section]
\newtheorem{defn}[theorem]{Definition}
\newtheorem{lem}[theorem]{Lemma}
\newtheorem{cor}[theorem]{Corollary}

\newtheorem{prop}[theorem]{Proposition}
\newtheorem{lemma}[theorem]{Lemma}

\newtheorem{proposition}[theorem]{Proposition}

\theoremstyle{definition}
\newtheorem{exm}[theorem]{Example}
\newtheorem{example}[theorem]{Example}
\newtheorem{definition}[theorem]{Definition}

\newtheorem{rem}[theorem]{Remark}

\newcommand{\ca}{\mathcal}

\newcommand{\bC}{\ensuremath{\mathbb{C}}}

\newcommand{\bS}{\ensuremath{\mathbb{S}}}

\newcommand{\bP}{\ensuremath{\mathbb{P}}}

\DeclareMathOperator{\cO}{\mathcal{O}}

\newcommand{\cA}{\ensuremath{\mathcal{A}}}
\newcommand{\cB}{\ensuremath{\mathcal{B}}}

\newcommand{\cE}{\ensuremath{\mathcal{S}}}
\newcommand{\cF}{\ensuremath{\mathcal{Q}}}

\newcommand{\cK}{\ensuremath{\mathcal{K}}}
\newcommand{\cS}{\ensuremath{\mathcal{S}}}

\newcommand{\cV}{\ensuremath{\mathcal{V}}}

\DeclareMathOperator{\cQ}{\mathcal{Q}}

\DeclareMathOperator{\codim}{codim}

\DeclareMathOperator{\Ext}{\mathrm{Ext}}
\DeclareMathOperator{\Hom}{\mathrm{Hom}}
\DeclareMathOperator{\Sym}{\mathrm{Sym}}

\usepackage{lipsum}

\DeclareMathOperator{\rank}{\operatorname{rank}}

\newcommand{\Gr}{\mathrm{Gr}}
\DeclareMathOperator{\quot}{\mathsf{Quot}}
\newcommand{\Quot}{\quot_d(\bP^1,V,r)}

\title{A Borel--Weil--Bott theorem for Quot schemes on $\bP^1$}

\author[A.~Gautam]{Ajay Gautam}
	\address{Department of Mathematics, Scuola Internazionale Superiore di Studi Avanzati, Trieste}
    \email{agautam@sissa.it}
	\author[F.~Lin]{Feiyang Lin}
	\address{Department of Mathematics, University of California, Berkeley}
	\email{fylin@berkeley.edu}
    \author[S.~Sinha]{Shubham Sinha}
	\address{The Abdus Salam International Centre for Theoretical Physics, Trieste}
    \email{ssinha1@ictp.it}

\begin{document}

\begin{abstract}
We study the cohomology groups of tautological bundles on Quot schemes over the projective line, which parametrize rank $r$ quotients of a vector bundle $V$ on $\bP^1$. Our main result is an analogue of the Borel--Weil--Bott theorem for Quot schemes. As a corollary, we prove recent conjectures of Marian, Oprea, and Sam on the exterior and symmetric powers of tautological bundles. 
\end{abstract}
\maketitle

\section{Introduction} 
Let $C$ be a smooth projective curve and $V$ a vector bundle of rank $n$ on $C$. Let $\quot_d(C,V,r)$ denote the Quot scheme that parameterizes short exact sequences of sheaves 
\[
0\to S\to V\to Q\to 0 \quad \text{with } \deg Q = d, \ \mathrm{rank}\, Q = r.
\]
Consider the universal exact sequence over the product $\quot_d(C,V,r) \times C$,
\begin{equation}\label{eq:universalSeq}
    0 \to \cE \to p^*V\to \cF \to 0
\end{equation}
where $p$ and $\pi$ denote the projection maps to $C$ and $\quot_d(C,V,r)$, respectively. Given vector bundles $K$ and $M$ on $C$, the associated tautological complexes on $\quot_d(C,V,r)$ are defined as the pushforwards
\[
K^{\{d\}} := R\pi_*\!\left(p^*K \otimes \cE \right), \qquad
M^{[d]} := R\pi_*\!\left(p^*M \otimes \cF \right).
\]
When $K^{\{d\}}$ and $M^{[d]}$ are vector bundles, which occurs when the bundles $K, M$ are sufficiently positive, we call them tautological bundles. 

The study of the Quot scheme $\quot_d(C,V,r)$ comes in two distinct flavors. When $r = 0$, i.e. punctual Quot schemes, $\quot_d(C,V,0)$ is a smooth projective variety of dimension $nd$, and $L^{[d]}$ is a vector bundle of rank $d$ for all line bundles $L$ on $C$. In this case, the geometry of the Quot scheme has close analogies with the Hilbert scheme of points on surfaces, manifesting, for instance, in the cohomological behavior of tautological bundles. A closed formula for the Euler characteristic of the exterior powers of $L^{[d]}$ was computed in \cite{Oprea-Shubham}, and a description of their cohomology groups was conjectured. This conjecture was proved in \cite{Marian-Oprea-Sam} for $C\cong\bP^1$. For arbitrary genus, it was partially proved in \cite{Krug}, and fully resolved in \cite{Marian-Negut} by studying the derived category of the Quot scheme.

When $r > 0$, many of the techniques available in the punctual case no longer apply, and not much is known about the cohomology of tautological bundles, except for some Euler characteristic calculations in \cite{Oprea-Shubham,SinhaZhang1,SinhaZhang2}. In this article, we study the cohomology of tautological complexes on $\quot_d(\bP^1, V, r)$ for any vector bundle $V$ and any $r$. Our theorems may be viewed as Borel–Weil–Bott–type results for Quot schemes on $\bP^1$ in direct analogy with the classical theorem for Grassmannians. 

\subsection{Cohomology of tautological bundles}
We begin by fixing notation. For a partition $\lambda=(\lambda_1,\dots,\lambda_k)$, let $|\lambda|$ denote the sum of its parts, and let $\lambda^\dagger$ be its conjugate partition. We write $\bS^{\lambda}$ for the Schur functor associated to $\lambda$; for example,
\[
\bS^{(\ell)}V = \Sym^{\ell} V \qquad \text{and} \qquad \bS^{(1)^\ell}V= \wedge^\ell V,
\]
where $(1)^\ell= (1, \dots, 1)$ with $\ell$ parts. 

Let \(V\) be a vector bundle on \(\mathbb{P}^1\) of rank $n$; fix \(r<n\). By \cite{Popa_Roth}, there exists an integer
\(d_0=d_0(V,r)\) such that for all \(d\ge d_0\), the Quot scheme
\[\Quot\quad(\text{abbreviated } \quot_d)\]
is irreducible and generically smooth of dimension \(nd + rb + r(n-r)\), where \(b=-\deg V\).
In the special case \(V= \cO_{\mathbb{P}^1}^{\oplus n}\), one may take \(d_0=0\) and \(\Quot\) is smooth and irreducible for all \(d\ge 0\). In the remainder of the introduction, we assume \(d \ge d_0(V,r)\).

We now state our main result.
\begin{theorem}\label{thm:intro_vb_insertion}
Let $K$ and $M$ be vector bundles on $\bP^1$, and let $\mu,\lambda$ be partitions satisfying
\[
|\mu|+|\lambda| \;<\; \frac{nd+rb+n}{\,n-r\,}.
\]
\begin{itemize}[leftmargin=2em]
  \item[(i)] If $\mu\neq \emptyset$, then
  \[
  H^i\bigl(\Quot,\, \bS^\mu K^{\{d\}} \otimes \bS^{\lambda} M^{[d]}\bigr)=0
  \qquad\text{for all } i\ge 0.
  \]
  \item[(ii)] If $\mu=\emptyset$, there is a natural isomorphism of graded vector spaces
  \[
  H^\bullet\!\bigl(\Quot,\, \bS^{\lambda} M^{[d]}\bigr)
  \;\cong\; \bS^{\lambda} H^\bullet(V\otimes M).
  \]
\end{itemize}
\end{theorem}
\begin{rem}
       When $M$ splits as a direct sum of line bundles of sufficiently positive degrees, the complex $M^{[d]}$, and hence $\bS^{\lambda} M^{[d]}$, is a vector bundle. The isomorphism in part (ii) of Theorem~\ref{thm:intro_vb_insertion}
\[
H^0(\Quot,\bS^{\lambda} M^{[d]}) \cong \bS^{\lambda} H^0(V \otimes M)
\]
is induced geometrically, by considering the morphism 
\[
R\pi_*p^*( V\otimes M) \;\longrightarrow\; M^{[d]}
\]
arising from the universal sequence \eqref{eq:universalSeq}, taking the associated morphism of the Schur functor and then applying the global sections functor. 
\end{rem}
Theorem~\ref{thm:intro_vb_insertion} greatly generalizes the results in \cite{Marian-Oprea-Sam}, even for punctual Quot schemes, and in the positive-rank setting, proves \cite[Conjecture~1.3.1,\ 1.3.2]{Marian-Oprea-Sam} concerning the exterior and symmetric powers of the tautological bundles. Specifically,
\begin{cor}
   Let $L$ be a line bundle on $\bP^1$. For all
$  k < (nd+rb+n)/(n-r)$, we have the following isomorphisms of graded vector spaces
\begin{itemize}[leftmargin=2em]
\item[(i)] $H^\bullet\!\left(\quot_d, \wedge^k L^{[d]}\right) 
  \cong \wedge^kH^\bullet(V\otimes L) $;
\item[(ii)] $H^\bullet\!\left(\quot_d, \Sym^k L^{[d]}\right) 
  \cong \Sym^k H^\bullet(V\otimes L).$
\end{itemize}
\end{cor}

\begin{rem}\label{rem:vb_to_multi_line}
    In Theorem~\ref{thm:multiple_sub_quot}, we prove a more general result analogous to Theorem~\ref{thm:intro_vb_insertion}, for the tensor product of Schur functors applied to multiple vector bundles $K_1, K_2, \dots, K_s$ and $M_{1}, M_2, \dots, M_t$. Note that while Theorem~\ref{thm:multiple_sub_quot} is stated for line bundles, it is easy to upgrade the theorem to allow vector bundles since every vector bundle splits on $\bP^1$. 
    Below, we state the result in the case of trivial $V$, describing each cohomology group explicitly.
\end{rem}
\begin{cor}\label{thm:multiple_insertion_intro}
Let $L_1, L_2, \ldots, L_t$ be line bundles on $\bP^1$. For any tuple of partitions $\lambda^1, \lambda^2, \ldots, \lambda^t$ satisfying $\lvert \lambda^1 \rvert + \lvert \lambda^2 \rvert + \cdots + \lvert \lambda^t \rvert < (nd+n)/(n-r)$, we have
    \[
    H^D \left(\quot_d(\bP^1,\cO_{\bP^1}^{\oplus n}, r),\bigotimes_{j = 1}^t \bS^{\lambda^j}L_j^{[d]}\right) \cong 
        \bigotimes_{ \deg L_j < 0} \bS^{(\lambda^j)^{\dagger}} H^1(L_j^{\oplus n})
        \otimes 
        \bigotimes_{\deg L_j \geq 0 } \bS^{\lambda^j} H^0( L_j^{\oplus n}),
    \]
    where $D = \sum_{\deg L_j < 0}|\lambda^j|$, and all other cohomology groups are zero. 
\end{cor}

Fixing the partitions and line bundles on $\bP^1$, we may view the corresponding tautological bundles on $\quot_d$ for any $d$, and study the asymptotic behavior of their cohomology groups as $d$ becomes large. With this perspective, we can view Theorem~\ref{thm:intro_vb_insertion} as a stabilization result. In particular, as predicted in \cite{Marian-Oprea-Sam},

\begin{cor}
For line bundles $L_1, L_2, \ldots, L_t$ and partitions $\lambda^1, \lambda^2, \ldots, \lambda^t$
    \[
\sum_{d=0}^{\infty} q^d\chi\left(\quot_d(\bP^1,\cO_{\bP^1}^{\oplus n},r),\bigotimes_{j=1}^{t} \bS^{\lambda^j}L_{j}^{[d]}\right)
\]
is given by a rational function with simple pole only possibly at $q=1$. 
\end{cor}

The vanishing of higher cohomology in Corollary~\ref{thm:multiple_insertion_intro} for line bundles of nonnegative degree can be established without any restriction on the sizes of the partitions, at the expense of imposing degree constraints on the line bundles involved. The next result is analogous to \cite[Theorem~3]{Lin2025}; see Theorem~\ref{prop: onlyH0-intro} for general $V$.
\begin{theorem}\label{thm:H0_no_bound_V_trivial}
     Let $L_1, L_2, \ldots, L_t$ be line bundles on $\bP^1$ of degrees at least $d$. For any tuple of partitions $\lambda^1, \lambda^2, \ldots, \lambda^t$,  we have
    \[
     H^i \left(\quot_d(\bP^1,\cO_{\bP^1}^{\oplus n},r),\bigotimes_{j = 1}^t \bS^{\lambda^j}L_j^{[d]}\right) = 0\quad \text{for all } i\ge1. 
    \]
\end{theorem}

The space of global sections in Theorem~\ref{thm:H0_no_bound_V_trivial} does not, in general, follow the description given in Corollary~\ref{thm:multiple_insertion_intro}. The size constraint on the partition $\lambda$ appearing in Theorem~\ref{thm:intro_vb_insertion} is quite close to optimal; see \cite[Theorem~12]{Oprea-Shubham} and Example~\ref{ex:size_constraints}. Theorem~\ref{thm:H0_no_bound_V_trivial} implies that the dimension of the space of global sections is equal to the Euler characteristic. The latter can be computed by other means, such as torus localization (see \cite[Theorem~3]{Oprea-Shubham} and \cite[Corollary~4.10]{SinhaZhang2} for exterior powers) or via wall-crossing to other moduli spaces \cite{DHL_FH}.

\subsection{Extension groups}
When the degree of a line bundle $L$ is sufficiently large, our method also allows the computation of the global extension groups for Schur functors of the vector bundle $L^{[d]}$. The answer is expressed in terms of Ext-groups of Schur functors of the universal quotient bundle on a Grassmannian (see Proposition~\ref{prop:Ext_groups_single}). The Borel--Weil--Bott theorem for the Grassmannian completely determines these Ext-groups and has the following implications.

\begin{theorem}\label{thm:Ext_implications}
       Assume $V$ splits as a direct sum of line bundles of nonpositive degrees. Let $L$ be a line bundle of degree at least $d+b$. Then 
    \begin{itemize}[leftmargin=2em]
        \item[(i)] For any nontrivial partition $\nu$ with $|\nu|<(nd+rb+n)/r$ and $\nu_1<n-r$,
    \begin{align*}
        H^i\left(\quot_d, (\bS^\nu L^{[d]})^\vee\right) = 
        0 \quad \text{for all } i \ge 0.
    \end{align*}
    \item[(ii)] Let $p_1,p_2,\dots,p_t$ and $k$ be nonnegative integers, and set $|p|=p_1+p_2+\cdots+p_t$.
Assume $k<(nd+rb+n)/(n-r)$, $|p|\le (nd+rb+n)/r$, and $t<n-r$. Then
    \[
    \Ext^i\left(\bigwedge^{p_1}L^{[d]}\otimes \cdots\otimes \bigwedge^{p_t}L^{[d]},\bigwedge^k L^{[d]}  \right)=\begin{cases}
        \bigwedge^{k-|p|}L^{[d]}& i=0,\ k\ge |p|;
        \\
        0& \text{otherwise}.
    \end{cases}
    \]
    A similar statement holds upon replacing all $\bigwedge $ with $\Sym$, and replacing the condition $t<n-r$ with $|p|<n-r$. \\
    \item[(iii)] For any partition $\lambda$, with $|\lambda|\le d$ and $\lambda_1<n-r$, 
    \[
    \Ext^i(\bS^\lambda L^{[d]},\bS^\lambda L^{[d]})=\begin{cases}
        \bC& i=0;\\
        0&i\ge 1.
    \end{cases}
    \]
    \item[(iv)] More generally, the set
    $
    \{\bS^{\lambda}(L^{[d]}): \ |\lambda|\le d ,\  \lambda_1<n-r\} 
    $
    forms an exceptional collection, ordered by the sizes of the partitions.
        \end{itemize}
\end{theorem}
\begin{rem}
   Note that the exceptional collection above generates only a subcategory of the derived category of coherent sheaves on $\quot_d$. A full exceptional collection built from Schur functors of universal quotients on Grassmannian was studied by Kapranov \cite{Kapranov}. For Quot schemes of rank-zero quotients, a full exceptional collection (and, in higher genus, a semiorthogonal decomposition) of the derived category was obtained in \cite{toda} and \cite{Marian-Negut}. However, these methods do not directly extend to the case of higher-rank quotients.
\end{rem}

Part~(ii) of Theorem~\ref{thm:Ext_implications} recovers the results on extension groups in \cite{Marian-Oprea-Sam} for Quot schemes of rank-zero quotients, and partially proves \cite[Conjecture~1.3.3]{Marian-Oprea-Sam} in the higher-rank setting. Moreover, in Theorem~\ref{thm:two_insertions}, we establish a mixed version of Proposition~\ref{prop:Ext_groups_single} and Theorem~\ref{thm:intro_vb_insertion} for two line bundles of consecutive degrees; this may be regarded as our main technical result. We state a consequence for the duals of tautological bundles below.

\begin{theorem} \label{thm:dual_two_insertion_intro}
   Assume $V$ splits as a direct sum of line bundles of nonpositive degrees. Let $L_{m-1}=\cO_{\bP^1}(m-1)$ and $L_{m}=\cO_{\bP^1}(m)$ on $\bP^1$, with $m\ge d+b$. For any pair of partitions $\nu$ and $\mu$ satisfying
\[
\lvert \nu\rvert + \lvert \mu\rvert \le (nd + rb + n)/r
\quad\text{and}\quad
0 < \nu_1 + \mu_1 < n - r,
\]
all cohomology groups of the vector bundle
$
\bigl(\bS^\nu L_{m-1}^{[d]}\bigr)^\vee \otimes \bigl(\bS^\mu L_m^{[d]}\bigr)^\vee
$
vanish.   
\end{theorem}

\subsection{Relation to quantum $K$-theory of Grassmannians}
When $V$ is the trivial vector bundle of rank $n$, the Quot scheme $\quot_d(\bP^1,V,r)$ is a smooth moduli space (see \cite{Stromme}) that compactifies the space of morphisms from $\bP^1$ to the Grassmannian $\Gr(r,n)$ of degree $d$. This viewpoint was taken up by Bertram \cite{Bertram} to study the quantum cohomology ring of $\Gr(r,n)$, and later to prove the Vafa–Intriligator formula \cite{Bertram2,Marian-Oprea} for intersection numbers on $\quot_d(C,V,r)$, which counts the number of maps from a fixed curve $C$ of genus $g$ to $\Gr(r,n)$ satisfying some incidence conditions.

Recently, \cite{SinhaZhang1,SinhaZhang2} studied \(K\)-theoretic invariants on \(\quot_d(\bP^1,V,r)\) in the context of the quantum \(K\)-ring of \(\Gr(r,n)\) (cf.~\cite{Buch_Mihalcea}). Fix a point \(x\in\bP^1\) and set
$$\cE_x:=\cE|_{\quot_d\times \{x\}},$$
a rank \(n-r\) vector bundle on \(\quot_d\). By \cite[Theorem~1.12]{SinhaZhang1}, the Euler characteristics of Schur functors of \(\cE_x\) compute \(K\)-theoretic Gromov–Witten invariants: For partitions \(\mu^1,\mu^2,\mu^3\) whose first parts are at most \(r\),
    \begin{equation}\label{eq:Gromov-Witten}
          \chi\left(\overline{\mathcal{M}}_{0,3}(\Gr(r,n),d),
	\bigotimes_{i\le 3} ev_{p_i}^*\bS^{\mu^{i}}(\cB)\right) = \chi\left(\quot_d(\bP^1,\cO_{\bP^1}^{\oplus n},r),\bigotimes_{i\le 3}
	\bS^{\mu^{i}}(\cE_x)\right)
    \end{equation}
Here \(\cB\) denotes the universal quotient bundle on \(\Gr(r,n)\), and \(\overline{\mathcal{M}}_{0,3}(\Gr(r,n),d)\) is the moduli space of \(3\)-pointed genus-\(0\) stable maps to \(\Gr(r,n)\) of degree \(d\) with \(ev_{p_i}\) the evaluation map at the \(i\)-th marking. In this paper, we prove the following consequence of Theorem~\ref{thm:intro_vb_insertion}.
 \begin{theorem}\label{cor:vanishing_E_x}
     Let $\mu^1,\mu^2,\dots\mu^t$ be partitions satisfying $0\ne|\mu^1|+\cdots+|\mu^t|<(nd+n)/(n-r)$, then
     $$H^i\left(\quot_d(\bP^1,\cO_{\bP^1}^{\oplus n},r),\bigotimes_{j=1}^t\bS^{\mu^j} \cE_x\right) =0\qquad\text{for all }i\ge 0.$$ 
 \end{theorem}
The Euler characteristics of Schur functors of $\cE_x$ were computed in \cite[Theorem~1.6 and 1.9]{SinhaZhang2}. The vanishing of Euler characteristics of these bundles was shown to provide enough relations to completely determine the quantum $K$-ring of $\Gr(r,n)$. Theorem~\ref{cor:vanishing_E_x} partially answers \cite[Question~1.11]{SinhaZhang2} by upgrading a subset of the vanishing statements for Euler characteristics to statements about individual cohomology groups. Using our results, one can already prove that the $K$-theoretic Gromov--Witten invariant in \eqref{eq:Gromov-Witten} vanishes for all sufficiently large degrees $d$. We expect that our methods may in future  shed light on positivity phenomena in the quantum $K$-theory of the Grassmannian; see, for instance, \cite{Buch_Mihalcea,BCMP}.

\subsection{Analogy with Hilbert scheme of points on surfaces}
Let $V$ be a vector bundle on a smooth projective curve $C$. The punctual Quot scheme, denoted $\quot_d(C,V)=\quot_d(C,V,0)$,  parametrizes quotients of $V$ supported at zero dimensional scheme of length $d$. Punctual Quot schemes have been extensively studied; we mention only a few relevant works here, including computations of cohomology groups \cite{marian_negut_coh}, descriptions of nef cones for divisors \cite{gangopadhyay-sebastian}, positivity results for tautological bundles \cite{oprea}, and cohomology of tangent bundle \cite{Biswas_Gangopadhyay_Sebatian}.

For a line bundle $L$ on $C$, the description of the cohomology groups for the tautological bundle $L^{[d]}$ and their exterior powers is explicitly given by the formula (see \cite[Theorem~3]{Marian-Negut} for the statement involving extension groups)
\[
H^\bullet(\quot_d(C,V),\wedge^k L^{[d]})\cong \wedge^kH^{\bullet}(V\otimes L)\otimes \mathrm{Sym}^{d-k}H^\bullet(\cO_C).
\]
There is a parallel story for the Hilbert schemes $X^{[d]}$ of $d$ points on a smooth projective surface $X$. For a line bundle $L$ on $X$, the tautological bundle $L^{[d]}$, defined similarly, has rank $d$. Indeed, the cohomology groups are given by the formula \cite{Scala_Coh_Hilb,Krug_McKay_corr}
\[
H^\bullet(X^{[d]},\wedge^k L^{[d]})\cong \wedge^kH^{\bullet}(L)\otimes \mathrm{Sym}^{d-k}H^\bullet(\cO_X).
\]
The proofs of the above statement in both cases rely on studying the bounded derived categories of sheaves $\mathbf{D}^b(X^{[d]})$ and $\mathbf{D}^b(\quot_d(C,V))$. 

The description of cohomology groups of the symmetric powers $\Sym^k L^{[d]}$, and general Schur functors, is not known to admit a simple formula, for both $X^{[d]}$ and $\quot_d(C,V)$. In special case, Euler characteristics \cite{Noah_Hilb_Scheme_KTheory} and the space of global sections \cite{Danila,LScala_symm} of the symmetric products of tautological bundles $L^{[d]}$ on $X^{[d]}$ admits a simple form. For instance, let $L$ be a positive degree line bundle on $X=\bP^2$, then for all $k<d+1$,
$$H^{0}\left(X^{[d]},\mathrm{Sym}^k\ L^{[d]}\right)=\Sym^kH^0(L)$$
while all higher cohomology groups vanish. We predict that an anolgue of formulas in Corollary~\ref{thm:multiple_insertion_intro} also hold for Hilbert scheme of points on $X\cong\bP^2$. We note here that tautological bundles on the Quot schemes on surfaces have also been studied in \cite{oprea-pandharipande,Arbesfeld_Johnson_lim_oprea_pandharipande}.

\subsection{Proof strategy}
For a fixed sufficiently large integer $m$, the Quot scheme on $\bP^1$ admits an embedding into a product of Grassmannians, constructed by Str\o mme \cite{Stromme},
\begin{equation*}
\iota_m:\ \quot_d \hookrightarrow \Gr(k_1,N_1)\times \Gr(k_2,N_2),
\end{equation*}
as the zero locus of a section of a vector bundle \(\cK\), expressed in terms of the universal bundles on the two Grassmannians. In the rank-zero setting, \cite{Marian-Oprea-Sam} uses the Koszul resolution of $\cO_{\quot_d}$ to reduce the study of the cohomology of exterior and symmetric powers of $L_m^{[d]}$ to the cohomology of universal bundles on the product of Grassmannians, where the Borel-Weil-Bott theorem applies.

We follow the approach of \cite{Marian-Oprea-Sam} to study the case of arbitrary rank, but the combinatorial analysis of the vector bundles that show up in the Koszul resolution becomes rather involved. To address this combinatorial proliferation, in Section~\ref{sec:Indices_of_partition} we provide a streamlined combinatorial criterion for detecting the non-vanishing of the cohomology of universal bundles on a Grassmannian, which may be of independent interest. A second ingredient is provided by Horn’s inequalities, which give criteria for the nonvanishing of the Littlewood–Richardson coefficients appearing in our expansions. By applying these two sets of criteria carefully, we show that all cohomology groups vanish for every term in the Koszul resolution except the first. We then conclude by computing the cohomology of the first term explicitly.

Our methods compute the cohomology of all Schur functors (subject to the relevant size constraints) of the tautological bundles associated to consecutive-degree line bundles, \(L_m^{[d]}\) and \(L_{m-1}^{[d]}\) (new even in the rank-zero case). This is especially advantageous because Schur functors of any tautological complex \(M^{[d]}\) can be expressed in terms of the Schur functors of the tautological bundles \(L_{m-1}^{[d]}\) and \(L_m^{[d]}\) in the derived category of sheaves on \(\quot_d\). We then leverage the explicit results about $\bS^{\alpha} L_{m-1}^{[d]}\,\otimes\, \bS^{\beta} L_{m}^{[d]}$ to obtain results for the Schur complexes \(\bS^{\lambda} M^{[d]}\), as well as their tensor products.

\noindent\textbf{Assumptions on characteristic.} Our results are valid for all algebraically closed fields of characteristic zero. For ease of notation, we use $\bC$ throughout. The main dependence on the characteristic comes from the use of the Borel-Weil-Bott theorem for Grassmannians. The fact that Schur functors of complexes respect quasi-isomorphisms, and therefore are well-defined on objects in the derived category, also depends on the characteristic zero assumption.

\subsection{Acknowledgements}
We thank Arvind Ayyer, David Eisenbud, Hannah Larson, Alina Marian, Leonardo Mihalcea, Noah Olander, Dragos Oprea, and Claudiu Raicu for helpful discussions. A.G. thanks the Simons Foundation for their support through the Joint SISSA/ICTP PhD fellowship.

\section{Combinatorial preliminaries}
\subsection{Notations} A partition $\lambda = (\lambda_1, \lambda_2, \dots, \lambda_k)$ is a non-increasing sequence of non-negative integers. The size of the partition is defined as $|\lambda| := \lambda_1 + \lambda_2 + \cdots + \lambda_k$. We can represent $\lambda$ graphically by its Young diagram, which contains $k$ rows of boxes with $\lambda_i$ boxes in the $i$th row. For example, the Young diagram of the partition $\lambda = (5,4,2,1)$, which has size $|\lambda| = 12$, is shown below:
$$
\ytableausetup{boxsize=1em}
    \begin{tikzpicture}
      \draw (0, 0)  node[anchor=north west] {%
        \ydiagram{5,4,2,1}
            *[*(lightgray!70)]{2,2}
      };
     \draw (-0.5, -0.8) node {$\lambda=$};
     \draw (5, 0)  node[anchor=north west] {%
        \ydiagram{4,3,2,2,1}
            *[*(lightgray!70)]{2,2}
      };
     \draw (4.5, -0.8) node {$\lambda^\dagger=$};
    \end{tikzpicture}
    $$
The \emph{conjugate partition} $\lambda^\dagger$ is obtained by transposing the Young diagram of $\lambda$. In the above example, $\lambda^\dagger = (4,3,2,2,1)$. 

The \emph{Durfee square} of a partition $\lambda$ is the largest square contained in the Young diagram of $\lambda$, and the side length of the Durfee square is called the \emph{rank} of the partition. In the above example, $\lambda$ has rank two.

\subsection{Schur functors}
Let $W \cong \mathbb{C}^k$ be a complex vector space. For any partition $\lambda = (\lambda_1, \lambda_2, \dots, \lambda_k)$, we denote by $\mathbb{S}^\lambda(W)$ the Schur functor, which corresponds to the irreducible polynomial $\mathrm{GL}(W)$-representation with highest weight $\lambda$. The character of the representation $\mathbb{S}^\lambda(W)$ is the Schur polynomial $s_\lambda(x_1, x_2, \dots, x_k)$, and its dimension is given by the hook-content formula:
\begin{equation*}
    \dim\bS^{\lambda}(W)=s_{\lambda}(\underbrace{1,1,\dots,1}_k) = \prod_{c\in \lambda}\frac{k+\mathrm{cont(c)}}{\mathrm{hook(c)}}
\end{equation*}
where for each cell $c=(i,j)$ in $\lambda$ is in the $i$th row and $j$th  column, $\mathrm{cont}(c) = j-i$ and $\mathrm{hook}(c)=1+(\lambda_i-j)+(\lambda_j^{\dagger}-i)$ is the hook length of $c$.

Every polynomial $\mathrm{GL}(W)$-representation can be expressed uniquely as a direct sum of Schur functors. Given two partitions $\lambda$ and $\mu$, one can decompose the tensor product as
\begin{equation}\label{eq:LRrule}
    \bS^{\lambda}(W) \otimes \bS^{\mu}(W) 
    = \bigoplus_\gamma \bS^{\gamma}(W)^{\oplus c_{\lambda, \mu}^\gamma},
\end{equation}
where the multiplicities $c_{\lambda, \mu}^\gamma$ are the \emph{Littlewood–Richardson coefficients}.  
The \emph{Littlewood–Richardson rule} provides a combinatorial description of these coefficients in terms of counting Littlewood–Richardson tableaux. In the next subsection, we list the properties of Littlewood–Richardson coefficients that we shall need in this article.
\\

Let $V$ be another complex vector space. Schur functors naturally appear in the context of \emph{Cauchy’s formula}: for any non-negative integer $t$, there is an isomorphism of vector spaces
\begin{equation}\label{eq:cauchyFormula}
     \bigwedge\nolimits^t (V \otimes W) 
     \cong \bigoplus_{\lvert \mu \rvert = t} \bS^{\mu}(V) \otimes \bS^{\mu^\dagger}(W),
\end{equation}
where the direct sum runs over all partitions $\mu$ of size $t$.

We will also require the following decomposition formula for the Schur functor of the direct sum of two vector spaces
\begin{equation}\label{eq: decompDirectSum}
   \bS^{\gamma}(V \oplus W) 
   \cong \bigoplus_{\alpha, \beta} 
   \left(\bS^{\alpha}(V) \otimes \bS^{\beta}(W)\right)^{\oplus c_{\alpha, \beta}^{\gamma}}.  
\end{equation}
Here the direct sum is taken over all partitions $\alpha$ and $\beta$ such that $|\alpha|+|\beta|=|\gamma|$.
\subsection{Schur functors associated to highest weights}\label{sec:highest_weight}
Irreducible rational representations of $\mathrm{GL}(W)$ are indexed by highest weights $$
\eta = (\eta_1, \eta_2, \dots, \eta_k),$$
which are non-increasing sequences of integers $\eta_1 \ge \eta_2 \ge \cdots \ge \eta_k$. Note that each $\eta_i$ is allowed to be negative.  
For any highest weight $\eta$, we denote the corresponding irreducible $\mathrm{GL}(W)$-representation by $\bS^{\eta}(W)$. For any $m\ge -\eta_k$, we can express $\bS^{\eta}(W)$ in terms of the usual Schur functor using the following isomorphism of $\mathrm{GL}(W)$-representations
\[
\bS^\eta W \cong \det(W)^{-m} \otimes \bS^{\eta + (m)^k}W,
\]
where $\det(W)$ is the determinant representation and $\eta + (m)^k = (\eta_1 + m, \eta_2 + m, \dots, \eta_k + m)$ is a partition. For any highest weight $\eta$, the dual of the corrresponding representation is given by $ \bS^{\eta} (W)^\vee = \bS^{-\eta}(W), $ where 
$$-\eta = (-\eta_k,-\eta_{k-1},\dots,-\eta_1).$$ 
The tensor product of two rational representations with highest weights $\eta$ and $\rho$ can also be described using the Littlewood-Richardson rule 
\begin{equation}\label{eq:LR_coeff_general}
    \bS^{\eta}W \otimes \bS^{\rho}W 
    = \bigoplus_\chi (\bS^{\chi}W)^{\oplus c_{\eta, \rho}^\chi},
\end{equation}
where $\chi$ runs over all the highest weights with $|\chi| = |\eta|+|\rho|$, and the multiplicity is given using the identity
\begin{align}\label{eq:modified_LR_coeff}
    c_{\eta,\rho}^{\chi} = c_{\eta+(m)^r,\rho+(k)^r }^{\chi +(m+k)^r},
\end{align}
where $m$ and $k$ are any two integers, such that $\eta+(m)^r,\rho+(k)^r$ and $\chi +(m+k)^r$ are partitions.

Given a highest weight $\eta = (\eta_1, \eta_2, \dots, \eta_k)$, it is sometimes useful to partition it into two parts, where one consists of nonnegative integers and one consists of nonpositive ones. For this, we use the notation $\eta = (\gamma, -\delta)$, where $\gamma, \delta$ are partitions. Note that due to the presence of zeros, there isn't necessarily a unique way to write $\eta$ in this form.
\begin{rem} 
Schur functors $\bS^{\eta}$ can be defined in the generality of vector bundles on a scheme. The identities in the above discussion, such as \eqref{eq:LRrule}, \eqref{eq:cauchyFormula}, and \eqref{eq: decompDirectSum}, all generalize verbatim. This is explained in e.g. \cite[Chapter 2]{weyman}. Schur functors associated to a partition can be generalized even further to perfect complexes, and it is well-defined up to quasi-isomorphisms. We explain the case of length two perfect complexes in Section \ref{sec:Schur_complex}.
\end{rem}

\subsection{Horn's inequalities}
There is a vast literature on the study of Littlewood--Richarson coefficients. In particular, a lot of it has been devoted to answering the following fundamental question: 
\begin{center}
   What are the conditions on the partitions $\alpha,\beta$ and $\gamma$ such that  $c_{\alpha,\beta}^{\gamma}\ne 0$?  
\end{center}
Knutson and Tao in their proof of the saturation conjecture \cite{KnutsonTao} showed that $c_{\alpha,\beta}^{\gamma} \ne 0$ if and only if $\alpha$, $\beta$, and $\gamma$ are the eigenvalues of Hermitian matrices $A$, $B$, and $C$ with $C = A + B$, respectively. We assume that the rank $N$ of the matrices to be greater than the number of parts in each of the partitions $\alpha$, $\beta$, and $\gamma$. Then using the identity for traces, $\mathrm{tr}A +\mathrm{tr}B=\mathrm{tr}C$, we get
\begin{equation*}
    \sum_{i=1}^{N} \alpha_i + \sum_{i=1}^{N} \beta_i = \sum_{i=1}^{N} \gamma_i.
\end{equation*}

In general, given $A$ and $B$ be $N\times N$ Hermitian matrices, define $C=A+B$. Let $\alpha=(\alpha_1\ge \cdots\ge \alpha_{N})$, $\beta=(\beta_1\ge \cdots \ge \beta_N)$ and $\gamma=(\gamma_1\ge \cdots\ge \gamma_N)$ denote the eigenvalues of the matrices $A$, $B$ and $C$ written in non-increasing order.

Let $I=(i_1\le \cdots\le i_s)$ and $J=(j_1\le \cdots\le j_s)$ be increasing sequence at most $N$ positive integers such that $i_s+j_s\le N+s$. Then we have
\begin{equation}\label{eq:horns}
    \sum_{i\in I}\alpha_i+\sum_{j\in J}\beta_j\ge \sum_{k\in K}\gamma_k
\end{equation}
where $K=(k_1\le \cdots\le k_s)$ with $k_p=i_p+j_p-p$ for $1\le p\le s$; see \cite{Fulton} for a detailed exposition.
 
In the following proposition, we enumerate various properties of Littlewood–Richardson coefficients that will be used later in the proofs.

\begin{prop}\label{prop:LR_coefficients_properties}
    Consider three partitions $\alpha$, $\beta$, and $\gamma$ such that $c_{\alpha,\beta}^\gamma \neq 0$. Then the following hold:
    \begin{enumerate}[label=(\roman*)]
        \item (Size constraint) 
        \[
        |\alpha| + |\beta| = |\gamma|;
        \]
        
        \item (Symmetry) 
        \[
        c_{\alpha, \beta}^{\gamma} = c_{\alpha^{\dagger}, \beta^{\dagger}}^{\gamma^{\dagger}};
        \]
        
        \item (Weyl inequality) For any positive integers $i$ and $j$, we have
        \begin{equation}\label{eq:WeylInequality}
                    \alpha_i + \beta_j \ge \gamma_{i+j-1};
        \end{equation}

        \item (Dominance inequality I) We have $\gamma\le \alpha+\beta$, that is, for all $s$,
        \begin{equation}\label{eq:dominance_inequality}
        \sum_{i=1}^s \gamma_i \le \sum_{i=1}^s \alpha_i + \sum_{i=1}^s \beta_i;
        \end{equation}

        \item (Dominance inequality II) We have $\alpha\cup\beta\le \gamma$, that is, for all $t$,
        \[
        \sum_{i=1}^t \alpha_i + \sum_{i=1}^t \beta_i \le \sum_{i=1}^{2t} \gamma_i.
        \]
    \end{enumerate}
Note that part (iii) and (iv) also holds when $\alpha,\beta$ and $\gamma$ are highest weights.
\end{prop}

\begin{proof}
    Statements (i) and (ii) are standard facts about Littlewood–Richardson coefficients. Statements (iii) and (iv) both follow from \eqref{eq:horns}. Specifically, statement (iii) can be obtained by letting $I = \{i\}, J = \{j\}$. Statement (iv) can be obtained by letting $I = J = \{1,\dots, s\}$. We now explain the proof of (v). Taking rank of matrices, $N$,  to be sufficiently large, the inequality in \eqref{eq:horns} implies 
    \[
    \sum_{p = t+1}^{t+s} \alpha_p + \sum_{p = t+1}^{t+s} \beta_p \ge \sum_{p = 2t+1}^{2t+s} \gamma_p
    \]
    for any $2s + t \le N + t$. Now take $s$ larger than the number of parts in $\alpha$, $\beta$, and $\gamma$. Then (v) follows by subtracting the above inequality from the equality in (i).
\end{proof}

\section{Indices of partitions}\label{sec:Indices_of_partition}
In this section, we focus on certain cases of the Borel-Weil-Bott theorem, and pictorially describe the combinatorial conditions on the partitions involved when a tautological bundle has a nontrivial cohomology group. The notions of the $t$-index and the $(t;\eta)$-index introduced in this section will be key ingredients in our proofs later.
\subsection{Borel-Weil-Bott}
Let $\Gr(k, N)$ denote the Grassmannian of $k$-dimensional subspaces of the vector space $\mathbb{C}^N$. Consider the universal exact sequence over $\Gr(k, N)$:
\[
0 \to \mathcal{A} \to \mathcal{O}_{\Gr(k, N)}^{\oplus N} \to \mathcal{B} \to 0,
\]
where $\mathcal{A}$ is the universal subbundle of rank $k$ and $\mathcal{B}$ is the universal quotient bundle of rank $N-k$. The sheaf cohomology of Schur functors applied to $\mathcal{A}$ and $\mathcal{B}$ is described by the Borel--Weil--Bott theorem \cite{Bott}. We state below the version presented in \cite[Corollary 4.1.9]{weyman}.

\begin{theorem}[Borel-Weil-Bott]\label{thm:BWB}
    For any two non-increasing sequences of integers
\[
\rho = (\rho_1, \dots, \rho_k) \quad \text{and} \quad \chi = (\chi_1, \dots, \chi_{N - k}),
\]
the vector bundle $\mathbb{S}^\rho\cA^\vee \otimes \mathbb{S}^\chi\cB^\vee$ on $\Gr(k, N)$ has at most one non-vanishing cohomology group. Let
\[
\omega := (\rho, \chi) + (N - 1, N - 2, \dots, 0)
\]
be the component-wise sum. If $\omega$ contains a repetition, then all cohomology groups of $\mathbb{S}^\rho(\cA^\vee) \otimes \mathbb{S}^\chi(\cB^\vee)$ vanish. Otherwise, let $\sigma$ be the permutation that sorts $\omega$ into a strictly decreasing sequence, and 
\[
\gamma := \sigma \circ \omega - (N - 1, N - 2, \dots, 0),
\]
be a non-increasing sequence, then
\begin{align*}
H^D\left(\Gr(k, N),\ \mathbb{S}^\rho\cA^\vee \otimes \mathbb{S}^\chi\cB^\vee \right)
=
\mathbb{S}^\gamma(\bC^N)^\vee
\end{align*}
where the cohomological degree $D=\ell(\sigma)$ is the length of the permutation $\sigma$, and all other cohomology groups are zero.
    \end{theorem}
\begin{rem}\label{rem:BWB_remark}
When we consider Schur functors of the quotient bundle independently, Theorem~\ref{thm:BWB} implies that for any partition $\lambda$, the morphism induced from the tautological sequence $$\bS^\lambda\cO_{\Gr(k,N)}^{\oplus N}\to \bS^\lambda \cB$$ gives an isomporhism on the space of global sections
$$
 H^0\left(\Gr(k, N),\ \bS^\lambda\cB\right)\cong \bS^\lambda(\bC^N).$$
Similar statement holds for $\bS^\lambda\cA^\vee$.
\end{rem}

\subsection{$t$-Index}
The notion of the $t$-index of a partition $\nu$ was introduced in \cite{Marian-Oprea-Sam} to systematically study the conditions which determine when all the cohomology groups of $\bS^\nu \cB^\vee$ vanish. Here we recall, and later generalize this notion which plays a crucial role in the sections that follow.
\begin{defn}
Fix a positive integer $t$. We say that a non-increasing sequence of integers $\chi=(\chi_1,\chi_2,\dots,\chi_m)$ has a \textbf{well-defined $t$-index} if there exists a nonnegative integer $j$ such that
\[
\chi_j \ge j + t \quad \text{and} \quad \chi_{j+1} \le j.
\]
In this case, the $t$-index of $\chi$ is $j$. Note that the $t$-index of $\chi$ is $0$ (resp.\ $m$) if $\chi_1\le 0$ (resp.\ $\chi_{m}\ge m+t$).
\end{defn}

\begin{rem}
 Note that every integer partition $\lambda$ has a well-defined $0$-index, which equals the rank of the partition. If $\lambda$ has a well-defined $t$-index $j$, then $j$ must also be the $0$-index (i.e., the rank) of $\lambda$. Furthermore, if $j$ is the $0$-index of $\lambda$, then $\lambda$ has a well-defined $t$-index if and only if $t \le \lambda_j - j$.
\end{rem}
\begin{exm}
  Fix $t = 3$. Consider the integer partitions $\lambda = (6,5,2,1)$ and $\mu = (7,4,2,2)$. The Durfee square of both partitions $\lambda$ and $\mu$ has side length $2$. We observe that $\lambda$ has a well-defined $t$-index, whereas $\mu$ does not.
    \[
      \ytableausetup{boxsize=0.8em}
    \begin{tikzpicture}
      \draw (0, 0)  node[anchor=north west] {%
        \ydiagram{6,5,2,1,}
            *[*(lightgray!70)]{2,2}
      };
     \draw (-0.5, -0.8) node {$\lambda=$};
    \end{tikzpicture}
    \quad \quad
    \ytableausetup{boxsize=0.8em}
    \begin{tikzpicture}
      \draw (0, 0)  node[anchor=north west] {%
        \ydiagram{7,4,2,2}
                    *[*(lightgray!70)]{2,2}
      };
             \draw (-0.5, -0.8) node {$\mu=$};
    \end{tikzpicture}
    \]
\end{exm}

\begin{lem}\label{lem:BWB_index}
    Let $\chi=(\chi_1,\dots,\chi_{N-k})$ be a non-decreasing sequence of integers. The vector bundle $\bS^{\chi}(\cB^\vee)$ on $\Gr(k,N)$ has a nontrivial cohomology group
    \begin{equation}\label{eq:Lemma_chi_BWB}
        H^D\!\big(\Gr(k,N), \bS^{\chi}(\cB^\vee)\big)\neq 0
    \end{equation}
    if and only if $\chi$ has a well-defined $k$-index $j$, and the cohomological degree is $D=kj$.
\end{lem}
\begin{proof}
    It is a straightforward application of Theorem~\ref{thm:BWB}. Indeed, \eqref{eq:Lemma_chi_BWB} is satisfied if and only if 
\begin{align*}
  \omega
  &= (0,\dots,0,\chi_1,\dots,\chi_{N-k}) + (N-1,\dots,N-k,\,N-k-1,\dots,0)\\
  &= (f_k,\dots,f_1,\,g_1,\dots,g_{N-k})
\end{align*}
has no repetitions, and $D$ is the length of the permutation that sorts $\omega$. Observe that $f_k>\cdots>f_1$ are consecutive integers, and thus we may choose $j$ such that
\[
g_1>\cdots>g_j>f_k>\cdots>f_1>g_{j+1}>\cdots>g_{N-k}.
\]
Hence the cohomological degree is $D=kj$. The required conditions $\chi_j \ge j+k$ and $\chi_{j+1}\le j$ follow from the inequalities $g_j>f_k$ and $g_{j+1}<f_1$, respectively. 
\end{proof}

\subsection{$(t; \eta)$-Index}
To pictorially describe the conditions for the vanishing of cohomology groups of the vector bundle $\mathbb{S}^\mu(\cA) \otimes \mathbb{S}^\eta(\cB)$, we introduce the following generalization of the notion of index.

\begin{defn}
Fix a positive integer $t$ and a tuple of non-increasing integers $\eta=(\eta_1,\dots,\eta_{r})$. Write $\eta=(\gamma,-\delta)$, where $\gamma$ and $\delta$ are partitions describing the nonnegative and negative entries of $\eta$.
We say that a partition $\mu$ has a \textbf{well-defined $(t; \eta)$-index} if there exists a non-negative integer $i$ such that
\begin{equation*}
    \mu_{i + 1 - s} \ge i + t - \gamma_s^\dagger \quad \text{and} \quad \mu_{i + s} \le i + \delta_s^{\dagger} \quad \text{for all } s \ge 1.
\end{equation*}
We call such an integer $i$ the $(t; \eta)$-index of $\mu$. 
\end{defn}
\begin{exm}
Fix $t = 3$ and $\eta = (1,-1,-1)$. Consider the partition $\nu = (6,4,3,1)$, and observe that $\nu$ also has a well-defined $(t;\eta)$-index, equal to $2$. We have $\delta = (1,1)$ and $\gamma = (1)$ in blue and red, respectively. 
    \end{exm}
\[
      \ytableausetup{boxsize=0.8em}
    \begin{tikzpicture}
      \draw (0, 0)  node[anchor=north west] (YD) {%
        \ydiagram{6,4,3,1}
            *[*(lightgray!70)]{2,2}
      };
      \fill[red!80,opacity=0.5] 
    ([xshift=3.75em,yshift=-1.2em]YD.north west)
    rectangle ++(0.8em,-0.9em);
     \fill[blue!80,opacity=0.5] 
    ([xshift=2.1em,yshift=-2em]YD.north west)
    rectangle ++(1.7em,-0.9em);
             \draw (-0.5, -0.8) node {$\nu=$};
       \draw (1.45, -1.4) node {\tiny $\delta_2$};
        \draw (1.1, -1.4) node {\tiny $\delta_1$};
        \draw (1.8, -1) node {\tiny $\gamma_1$};
    \end{tikzpicture}
    \]
    However, $\nu$ itself does not have a well-defined $t$-index. This example also illustrates that the $(t; \eta)$-index of a partition is not necessarily equal to the rank of the partition.

\begin{lem}\label{lem:BWB_index_multi}
Let $\mu = (\mu_1, \dots, \mu_k)$ be a partition and $\eta=(\eta_1,\eta_2,\dots,\eta_{N-k})$ be tuple of non-increasing integers. Write $\eta = (\gamma,-\delta)$ for partitions $\gamma$ and $\delta$. Suppose that the vector bundle $\bS^\mu(\cA) \otimes \bS^\eta(\cB)$ on $\Gr(k,N)$ has a (unique) nonzero cohomology group, namely $$H^D(\Gr(k,N), \bS^\mu(\cA) \otimes \bS^\eta(\cB))\ne 0.$$ Then the following hold.
\begin{itemize}[leftmargin=2em]
    \item[(i)] The partition $\mu$ has a well-defined $(N - k; \eta)$-index, i.e., there exists a natural number $i$ such that
\begin{equation}\label{eq:(t,eta)-index}
    \mu_{i + 1 - s} \ge i + N - k - \gamma_s^\dagger \quad \text{and} \quad \mu_{i + s} \le i + \delta_s^\dagger \quad \text{for all } s \ge 1.
\end{equation}
\item[(ii)] The cohomological degree satisfies the bound
\[
D \le |\delta| + (\mu_1+\mu_2+\cdots+\mu_i) -i^2.
\]
\end{itemize}
\end{lem}

\begin{proof}
Let $\eta=(\gamma_1,\dots,\gamma_q,-\delta_p,\dots,-\delta_1)$ be a tuple of non-increasing integers, where $\gamma$ and $\delta$ are partitions with at most $q$ and $p$ parts respectively, and $p+q=N-k$.
We will show that $\mu$ has a well-defined $(N-k;\eta)$-index $i$ given by the condition
\begin{equation}\label{eq:index_lem}
  \mu_i \ge i+p \quad \text{and} \quad \mu_{i+1} < i+p .
\end{equation}
Note that $\bS^{\mu}\cA \cong \bS^{-\mu}\cA^{\vee}$ and $\bS^\eta\cB \cong \bS^{-\eta}(\cB^\vee)$, where $-\eta=(\delta,-\gamma)$. Theorem~\ref{thm:BWB} implies that the vector
\[
  \omega
  = (-\mu_{k}, \ldots, -\mu_1, \delta_1, \ldots, \delta_p, -\gamma_q, \ldots, -\gamma_1) + (N-1, \ldots, 1, 0)
\]
has no repetition, and that $D$ is the length of the permutation that sorts $\omega$. We split the analysis into two parts, first looking at the entries of $\omega$ that are strictly less than $q$, and then at those that are at least $q$.

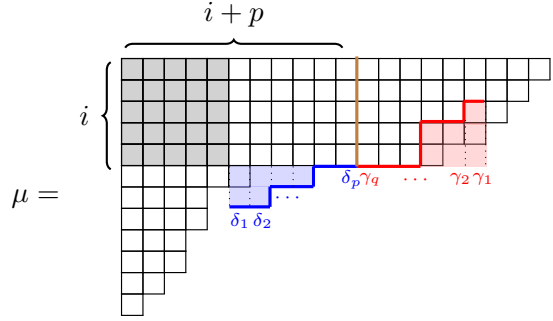
\begin{equation*}
\ytableausetup{boxsize=0.7em}
\begin{tikzpicture}
      \draw (0, 0) node[anchor=north west] (YD) {%
        \ydiagram{20,19,18,17,15, 6, 4, 4,3, 3,2,1}
        *[*(lightgray!70)]{5,5,5,5,5}
      };
\fill[red!50,opacity=0.3] 
    ([xshift=12.2em,yshift=-1.9em]YD.north west)
    rectangle ++(0.75em,-0.75em);
\fill[red!50,opacity=0.3] 
    ([xshift=10.7em,yshift=-2.6em]YD.north west)
    rectangle ++(2.3em,-1.5em);

            \draw[red, very thick]
  ([xshift=12.2em,yshift=-1.85em]YD.north west) -- 
  ++(0.7em,0);
 \draw[red, very thick]
  ([xshift=12.2em,yshift=-1.85em]YD.north west) -- 
  ++(0,-0.7em);
  \draw[red, very thick]
  ([xshift=10.7em,yshift=-2.55em]YD.north west) -- 
  ++(1.5em,0);
  \draw[red, very thick]
  ([xshift=10.7em,yshift=-2.55em]YD.north west) -- 
  ++(0,-1.5em);
   \draw[red, very thick]
  ([xshift=8.5em,yshift=-4.1em]YD.north west) -- 
  ++(2.2em,0);

\draw[black, dotted]([xshift=12.95em,yshift=-1.8em]YD.north west) -- 
  ++(0,-2.3em);
  \draw[black, dotted]([xshift=12.25em,yshift=-2.5em]YD.north west) -- 
  ++(0,-1.6em);

      \fill[blue!50,opacity=0.3] 
    ([xshift=4.1em,yshift=-4.1em]YD.north west)
    rectangle ++(2.8em,-0.7em);
          \fill[blue!50,opacity=0.3] 
    ([xshift=4.1em,yshift=-4.8em]YD.north west)
    rectangle ++(1.4em,-0.7em);
    \draw[blue, very thick]
  ([xshift=7em,yshift=-4.1em]YD.north west) -- 
  ++(1.5em,0em);
            \draw[blue, very thick]
  ([xshift=7em,yshift=-4.1em]YD.north west) -- 
  ++(0,-0.7em);
            \draw[blue, very thick]
  ([xshift=5.5em,yshift=-4.8em]YD.north west) -- 
  ++(1.5em,0em);
  
              \draw[blue, very thick]
  ([xshift=5.5em,yshift=-4.8em]YD.north west) -- 
  ++(0,-0.7em);
          \draw[blue, very thick]
  ([xshift=4.1em,yshift=-5.5em]YD.north west) -- 
  ++(1.4em,0em);
\draw[black,dotted]([xshift=5.55em,yshift=-4.1em]YD.north west) -- 
  ++(0em,-0.7em);
  \draw[black,dotted]([xshift=6.3em,yshift=-4.1em]YD.north west) -- 
  ++(0em,-0.7em);
\draw[black,dotted]([xshift=4.8em,yshift=-4.1em]YD.north west) -- 
  ++(0em,-1.5em);
  \draw[black,dotted]([xshift=4.1em,yshift=-4.1em]YD.north west) -- 
  ++(0em,-1.5em);

    \draw (-1, -2) node {$\mu=$};
            \draw (4.95, -1.75) node {\tiny $\textcolor{red}{\gamma_1}$};
         \draw (4.65, -1.75) node {\tiny $\textcolor{red}{\gamma_2}$};
         \draw (4.05, -1.75) node {\tiny $\textcolor{red}\cdots$};
           \draw (3.45, -1.75) node {\tiny $\textcolor{red}{\gamma_q}$};

         \draw (3.2, -1.75) node {\tiny $\textcolor{blue}{\delta_p}$};
         \draw (2.35, -2) node {\tiny $\textcolor{blue}{\cdots}$};
         \draw (2.02, -2.3) node {\tiny $\textcolor{blue}{\delta_2}$};
         \draw (1.72, -2.3) node {\tiny $\textcolor{blue}{\delta_1}$};

    \draw[decorate,decoration={brace,amplitude=3pt},thick]
    ([xshift=0.5em,yshift=0em]YD.north west) -- ++(7.5em,0)
    node[midway,above=4pt] {$i+p$};
    
    \draw[brown, very thick]([xshift=8.5em,yshift=-4.1em]YD.north west) -- 
  ++(0em,3.8em);
\draw[decorate,decoration={brace,mirror,amplitude=3pt},thick]
    ([xshift=-0em,yshift=-0.5em]YD.north west) -- ++(0,-3.7em)
    node[midway,left=4pt] {$i$};
    \end{tikzpicture}
\end{equation*}

By the definition of $i$ in \eqref{eq:index_lem}, the entries strictly less than $q$ in $\omega$ consist of
\[
  p+q+(i-1)-\mu_i,\ \ldots,\ p+q-\mu_1 \quad \text{and} \quad q-1-\gamma_q,\ \ldots,\ -\gamma_1.
\]
Subtracting $q-1$ from each of the entries above and negating,  we may assume that there are no repetitions among the following $i+q$ nonnegative integers
\[
  \alpha_i,\alpha_{i-1}+1,\ \ldots,\ \alpha_1+i-1 \quad \text{and} \quad \gamma_q,\gamma_{q-1}+1,\ \ldots,\ \gamma_1+q-1,
\]
where $\alpha=(\mu_1-p-i,\ldots,\mu_i-p-i)$. Lemma~\ref{lem:Beta_set} implies that $\alpha_{i-s}\ge q-\gamma_s^\dagger$ for all $s\ge 0$, which proves the first set of inequalities in \eqref{eq:(t,eta)-index}. Moreover, the length of the permutation that sorts the entries above (in the stated order) is at most $|\alpha|$, hence the contribution to the length of the permutation that sorts $\omega$ is at most $ip+|\alpha|$ (extra $ip$ for moving the $i$ entries containing $\mu_1,\dots,\mu_i$ to the right of entries containing $\delta_1,\dots,\delta_p$ in $\omega$).

Now we consider the entries of $\omega$ that are at least $q$, namely
\[
  \delta_p,\delta_{p-1}+1,\ \ldots,\ \delta_1+p-1 \quad \text{and} \quad \beta_{k-1},\beta_{k-2}+1,\ \ldots,\ \beta_1 + (k-i)-1,
\]
where $\beta=(i+p-\mu_{k},\ldots,i+p-\mu_{i+1})$. Lemma~\ref{lem:Beta_set} implies $\beta_{i-s} \ge p-\delta_s$, proving the second set of inequalities in \eqref{eq:(t,eta)-index}. Furthermore, the length of the permutation that sorts the entries above contributes at most $|\delta|$ to the length of the permutation that sorts $\omega$. Hence $D \le |\delta| + ip + |\alpha| = |\delta|+(\mu_1+\cdots+\mu_i)-i^2$.
\end{proof}

The following result is elementary; due to the lack of a reference, we provide a short proof for completeness. We use the abacus notation for partitions (also called beta-sets in the literature cf. \cite{kerberjames}). We adapt the proof from \cite[(1.7)]{Macdonald} to our setting.
\begin{lem}\label{lem:Beta_set}
    Let $\alpha=(\alpha_1,\alpha_2,\dots,\alpha_i)$ and $\lambda=(\lambda_1,\lambda_2,\dots,\lambda_q)$ be partitions with at most $i$ and $q$ parts respectively. If the string of natural numbers
    \[
    \omega=(\alpha_i,\alpha_{i-1}+1,\dots,\alpha_1+i-1,\lambda_q,\lambda_{q-1}+1,\dots,\lambda_1+q-1)
    \] 
    has no repitition, then the permutation $\sigma$ that sorts $\omega$ has length $\ell(\sigma)\le |\alpha|$, and $$\alpha_{i-s}\ge q-\lambda_s^\dagger\quad \text{for all } 0\le s<i. $$
\end{lem}
\begin{proof}
 We define a sequence $u=(u_0,u_1,u_2,\dots)$ of $0$'s and $1$'s, by
\[
u_{m}=\begin{cases}
    1&\text{if } m\in \{\alpha_i,\dots,\alpha_1+i-1\}\\
    0&\text{otherwise}.
\end{cases}
\]
Similarly define $v=(v_0,v_1,v_2,\dots)\in \{0,1\}^{\mathbb{N}}$ with $1$'s placed exactly at the positions $(\lambda_q,\dots,\lambda_1+q-1)$. Note that $\omega$ has no repetition if and only if $u+v\in \{0,1\}^\mathbb{N}$. The length of permutation, by definition, equals
\begin{align*}
\ell(\sigma) &= \#\{ (m,n )\in \mathbb{N}^2: n<m,\ v_{n}=1, \ u_{m}=1\}
\\
    &\le \#\{ (m,n )\in\mathbb{N}^2:  n<m,\ u_{n}=0 ,\ u_{m}=1\} = |\alpha|.
\end{align*}
Let $\hat{v}$ be the sequence obtained by flipping 1's and 0's in $v$. To observe the inequalities, we will draw the lattice path diagrams for $u$ and $\hat{v}$, such that $0$'s denote increment in the $y$-coordinate by one, and $1$'s  denote right step by increasing $x$-coordinate by one. The condition $u+v\in \{0,1\}^{\mathbb{N}}$ implies that for any $m\ge 0$, 
\[
\text{number of 0's in }(u_0,\dots,u_m)\, \ge\, \text{number of 0's in }(\hat{v}_0,\dots,\hat{v}_m),
\]
and hence the lattice path of $\hat{v}$ lies always below that of $u$.
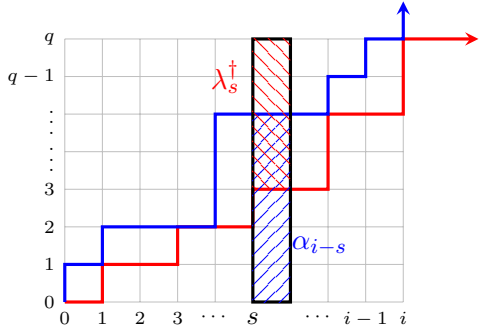
\[
\begin{tikzpicture}[scale=0.5,>=stealth]
  \draw[step=1cm,lightgray,very thin] (0,0) grid (9,7);
  \foreach \x in {0,...,3} \node[below] at (\x,0) {\tiny $\x$};
\foreach \x in {4,...,4} \node[below] at (\x,0) {\tiny $\cdots$};
\foreach \x in {5,...,5} \node[below] at (\x,0) {\small $s$};
\foreach \x in {7,...,7} \node[below] at (\x,0) {\tiny $\cdots\hspace{1em}$};
\foreach \x in {8,...,8} \node[below] at (\x,0) {\tiny $i-1$};
\foreach \x in {9,...,9} \node[below] at (\x,0) {\tiny $i$};
  
  \foreach \y in {0,...,3} \node[left]  at (0,\y) 
  {\tiny $\y$};
   \foreach \y in {4,...,5} \node[left]  at (0,\y) 
  {\tiny $\vdots$};
   \foreach \y in {6,...,6} \node[left]  at (0,\y) 
  {\tiny $q-1$};
  \foreach \y in {7,...,7} \node[left]  at (0,\y) 
  {\tiny $q$};


  \draw[red,very thick,-stealth]
    (0,0) -- (1,0) -- (1,1) -- (2,1) -- (3,1) -- (3,2) -- (4,2) -- (5,2)--(5,3)--(7,3)--(7,4)--(7,5)--(9,5)--(9,7)--(11,7);

  \draw[blue,very thick,-stealth]
    (0,0) -- (0,1) -- (1,1) -- (1,2) -- (2,2) -- (4,2) -- (4,4) -- (4,5) -- (5,5) -- (7,5) -- (7,6)--(8,6)--(8,7)--(9,7)--(9,8);

\draw[black,very thick] (5,0) rectangle (6,7);
\fill[pattern={Lines[angle=45,distance=4pt]},pattern color=blue]
       (5,0) rectangle (6,5); 
\node[anchor=west, xshift=1em] at (5,1.5) {$\textcolor{blue}{\alpha_{i-s}}$};

\fill[pattern={Lines[angle=135,distance=4pt]},pattern color=red]
       (5,3) rectangle (6,7); 
\node[anchor=west, xshift=-0.5em] at (4,6) {$\textcolor{red}{\lambda_{s}^{\dagger}}$};
\end{tikzpicture}
\]
By construction, the lattice path of $u$ (denoted in blue color) carves out $\alpha$ over $x$-axis, with $\alpha_{i-s}$ be the $y$-coordinate for the $s$-th right step. Similarly, the lattice path of $\hat{v}$ (colored in red) carves out $\lambda$ over $y$-axis. The inequality follows easily by considering the rectangle $[s,s+1]\times [0,q]$ in the $i\times q$ grid as showing in the sketch.
\end{proof}

\section{Two line bundles of consecutive degrees}
Our main goal in this
section is to study the cohomology groups of tensor products of Schur functors applied to the
tautological vector bundles $L_{m-1}^{[d]}$ and $L_{m}^{[d]}$ on $\quot_d(\bP^1,V,r)$, for line bundles
\[
L_{m-1}:=\cO_{\bP^1}(m-1)\quad\text{and}\quad L_m:=\cO_{\bP^1}(m)
\]
when $m$ is sufficiently large. We will prove slightly more general statement than what is stated in the introduction for Schur functors corresponding to highest weights (see Section~\ref{sec:highest_weight}) instead of partitions. The statement is of independent interest, and we state it below.

Let us first recall the notations. We consider Quot scheme $\Quot$, $\quot_d$ for short, for the vector bundle $V$ on $\bP^1$, which splits as
\[
V \cong \cO(-b_1)\oplus \cO(-b_2)\oplus \cdots \oplus \cO(-b_n)
\]
where $0=b_1\le b_2\le \cdots \le b_n$, and thus the sum $b=\sum_{i=1}^{n}b_i$ is a nonnegative number. Furthermore, we will assume that $d\ge d_0(V,r)$ such that $\quot_d$ is an irreducible variety of dimension $nd+rb+r(n-r)$ for the rest of this article. 
\begin{theorem}\label{thm:two_insertions}
    Let $\eta = (\gamma,-\delta)$ and $\rho=(\lambda,-\nu)$ be non-increasing sequences of lengths $\rank L_{m-1}^{[d]}$ and $\rank L_{m}^{[d]}$ respectively. Suppose $\gamma,\lambda,\nu,\delta$ are partitions satisfying
    $$(n-r)(\lvert \lambda \rvert+|\gamma|) + r(|\nu|+\lvert\delta \rvert) < nd+rb+n;\qquad \nu_1+\delta_1 < n-r.$$
    (i) If either $|\delta|\ne 0$ or $|\nu|\ne 0$, and $m\ge b+d$, then
    $$H^i\left(\quot_d,\bS^{\eta} L_{m-1}^{[d]}\otimes \bS^{\rho} L_{m}^{[d]}\right)=0\quad \text{for all }i\ge 0.$$
    (ii) If $\delta=\nu=\emptyset$, then, for $m$ sufficiently large,
    \begin{align}
        H^0\left(\quot_d, \bS^\gamma L_{m-1}^{[d]}\otimes \bS^\lambda L_{m}^{[d]}  \right) \cong \bS^{\gamma}H^0(V\otimes L_{m-1})\otimes \bS^{\lambda}H^0(V\otimes L_m)
    \end{align}
    and all higher cohomology groups vanish.
\end{theorem}
The condition for $m$ to be sufficiently large in part(ii) of Theorem~\ref{thm:two_insertions} will be removed in the next section. The proof of Theorem~\ref{thm:two_insertions} is quite involved, its components will be split in several subsections below. Using the same method, we will also obtain a bound on the cohomological degree of $\bS^{\eta} L_{m-1}^{[d]}\otimes \bS^{\rho} L_{m}^{[d]}$ in Proposition~\ref{prop:cohomological_degree_no_size_constraints} that holds with no restrictions on $\nu$ and $\rho$.
\begin{rem}
The condition $b_1=0$ is not an extra assumption, since there is an isomorphism
$$
\quot_d(\bP^1,V,r)\;\cong\; \quot_{\,d+rb_1}\!\big(\bP^1, V(b_1), r\big),
$$
which takes $L_{m}^{[d]}$ to $L_{m-b_1}^{[d+rb_1]}$ on the right, which reduces the analysis to our cases.
\end{rem}
\begin{rem}\label{rem:structure_sheaf}
    The special case when $\eta$ and $\rho$ are trivial in Theorem~\ref{thm:two_insertions} implies that $$H^0(\quot_d,\cO_{\quot_d}) = \bC,$$ while all higher cohomology groups vanish. This is consistent with the fact that $\quot_d(\bP^1,\cO_{\bP^1}^{\oplus n}, r)$ is a rational variety \cite{Stromme}.  
\end{rem}

\subsection{Str\o mme embedding}\label{sec:Stromme_embedding}
There is an embedding of the Quot scheme into a product of two Grassmannians 
(see \cite{Stromme} for details):
\begin{equation}\label{eq:Stromme_iota_sec}
   \iota:\quot_d(\bP^1, V,r)\hookrightarrow \Gr(k_1,N_1)\times \Gr(k_2,N_2).
\end{equation}
The explicit values of $k_1,k_2,N_1,$ and $N_2$ depend on the choice of an 
integer $m \ge b+d$, and are given as follows:
    \begin{align*}
        N_1 &= H^0(V(m-1))=nm - b,
                &N_2 &=H^0(V(m))= n(m+1) - b, \nonumber\\ 
                 r_1 &=\rank(L_{m-1}^{[d]})= rm + d,  
&r_2&=\rank(L_m^{[d]})= r(m+1) + d,\\
k_1&= \rank(L_{m-1}^{\{d\}}) = (n-r)m-b-d, & k_2&=\rank(L_{m}^{\{d\}})=  (n-r)(m+1)-b-d.\nonumber
    \end{align*}
We give a short overview of the construction. Let us recall the universal exact sequence
\[ 0 \to \cE \to p^*V \to \cF \to 0,\] 
on $\quot_d(\bP^1, V,r) \times \bP^1$ where $\pi$ and $p$ are projections to $\quot_d(\bP^1,V,r)$ and $\bP^1$ respectively. Twisting the above short exact sequence by $p^*L_m$ and pushing forward, we get an exact sequence of vector bundles (since $m\ge b+d$)
\begin{equation}\label{eq:pushFwdTwistedUniversalSeq}
   0 \to \pi_*(\cE(m)) \to H^0(V(m)) \otimes 
\cO_{\quot_{d}} \to \pi_*(\cF(m)) \to 0.
\end{equation}
This gives us a morphism from $\iota_m \colon \quot_d(\bP^1, V,r) \to \Gr(k_2,N_2)$.
Similarly, twisting by $p^*L_{m-1}$ yields an analogous sequence of 
vector bundles as in \eqref{eq:pushFwdTwistedUniversalSeq} and hence a morphism $\iota_{m-1} \colon \quot_d(\bP^1, V,r) \to \Gr(k_1,N_1)$. The Str\o mme embedding in \eqref{eq:Stromme_iota_sec} is defined as $\iota:= (\iota_{m-1}, \iota_m)$.

We have universal short exact sequences 
\begin{equation*}
    0 \to \cA_i \to W_i \to \cB_i \to 0
\end{equation*}
on each $\Gr(k_i,N_i)$ for $i=1,2$ where 
$$W_1 := H^0(V(m-1))\otimes\cO_{\Gr(k_1,N_1)} \quad \text{and}\quad W_2=H^0(V(m))\otimes \cO_{\Gr(k_2,N_2)}. $$
By a slight abuse of notation, we use the same notation for the pullbacks of these bundles to the product $\Gr(k_1,N_2)\times \Gr(k_2,N_2)$ along the projection maps.

Consider the section $\cO_{\bP^1} \to \cO_{\bP^1}(1) \boxtimes \cO_{\bP^1}(1)$ defining the diagonal $\Delta \subset \bP^1 \times \bP^1$. Taking the tensor product  with the pullback of $V(m-1)$ along the first factor of $\bP^1$ and passing to cohomology, we get an inclusion map
\[ W_1 \to W_2 \otimes H^0(\cO_{\bP^1}(1)).\]
Consider the following diagram on $\Gr(k_1,N_1)\times \Gr(k_2,N_2)$,
\[
\begin{tikzcd}
    0 \arrow[r] & \cA_1 \arrow[r] & W_1   \arrow[r] \arrow[d] & \cB_1 \arrow[r] & 0\\
    0 \arrow[r] & \cA_2 \otimes H^0(\cO_{\bP^1}(1)) \arrow[r] & W_2 \otimes  H^0(\cO_{\bP^1}(1))  \arrow[r] & \cB_2  \otimes H^0(\cO_{\bP^1}(1)) \arrow[r] & 0.
\end{tikzcd}
\]
The composition $$\cA_1 \to W_1  \to  W_2 \otimes  H^0(\cO_{\bP^1}(1)) \to \cB_2 \otimes H^0(\cO_{\bP^1}(1))$$ gives us a section $s$ of the vector bundle
\[\ca{K}:= \cA_1^\vee \otimes \cB_2 \otimes H^0(\cO_{\bP^1}(1)).\]

Following the proof of \cite[Thm.~4.1]{Stromme} with minor modifications, one shows that $\iota$ is a closed embedding. Moreover, the image $\iota(\quot_d)$ is precisely the zero locus of the section $s$ defined above. We note here that 
\[L_{m-1}^{[d]} \cong \iota^*(\cB_1) \quad \text{ and } \quad  L_{m}^{[d]} \cong \iota^*(\cB_2).\]
Similarly, we have
\[
L_{m-1}^{\{d\}} \cong \iota^*(\cA_1) \quad \text{ and } \quad  
L_{m-1}^{\{d\}} \cong \iota^*(\cA_2).\]

\subsection{Proof of Theorem~\ref{thm:two_insertions}}\label{sec:Kozsul_resolution}
Let $\iota: \Quot \to \Gr(k_1, N_1) \times \Gr(k_2, N_2)$ be the Str\o mme embedding associated to the positive integer $m$, as in Section~\ref{sec:Stromme_embedding}.
As discussed in the previous section, $\iota(\quot_d)$ is the vanishing locus of a regular section $s$ of the vector bundle $\ca{K}$. The assumption $d\ge d_0(V,r)$ implies that $\quot_d$ is a local complete intersection of expected dimension
\[
\dim(\Gr(k_1,N_1)\times \Gr(k_2,N_2))-\rank \cK = nd+rb+r(n-r).
\]
We thus have the Koszul resolution
\begin{equation}\label{eqn: koszul_res} 
    \begin{tikzcd}
    \cdots \arrow[r] & \bigwedge^2 \ca{K}^\vee \arrow[r] & \ca{K}^\vee \arrow[r] & \cO_{\Gr(k_1,N_1) \times \Gr(k_2,N_2)} \arrow[r] & \iota_* \cO_{\quot_d} \arrow[r] & 0.
\end{tikzcd}
\end{equation}
We can apply formulas \eqref{eq:cauchyFormula} and \eqref{eq: decompDirectSum} to decompose $\bigwedge^t \ca{K}^\vee$ as
\begin{align*}
    \bigwedge^t \ca{K}^\vee &= \bigwedge^t (\cA_1 \boxtimes (\cB_2^\vee \oplus \cB_2^\vee)\\
    &= \bigoplus_{\lvert \mu \rvert = t} \bS^{\mu}\cA_1 \boxtimes \bS^{\mu^\dagger}(\cB_2^\vee \oplus \cB_2^\vee)\\
    &= \bigoplus_{\lvert \mu \rvert = t} \bS^{\mu}\cA_1 \boxtimes \left( \bigoplus_{\alpha,\beta, \sigma} (\bS^\sigma \cB_2^\vee )^{\oplus c_{\alpha,\beta}^{\mu^\dagger}\cdot c_{\alpha, \beta}^\sigma} \right).
\end{align*}
This shows that the direct summands of $\bigwedge^t \ca{K}^\vee$ are of the form
\[ \bS^\mu \cA_1 \boxtimes \bS^\sigma \cB_2^\vee \quad \text{with } \lvert \mu \rvert = t \text{ and there exist $\alpha, \beta$ such that } c_{\alpha,\beta}^{\mu^\dagger}\cdot c_{\alpha,\beta}^\sigma \neq 0.\]
\begin{proof}[Proof of Theorem~\ref{thm:two_insertions}]
    The tautological bundles in consideration can be expressed as a pullback via \eqref{eq:Stromme_iota_sec}
    \[
     \bS^\eta L_{m-1}^{[d]} = \iota^* \bS^\eta\cB_{1} \quad \text{and}\quad \bS^\rho L_{m}^{[d]} = \iota^* \bS^\rho\cB_{2}.
    \]
      Tensoring the Koszul resolution in \eqref{eqn: koszul_res} with $S^{\eta} \cB_1\boxtimes \bS^{\rho}\cB_{2}$, we obtain
\begin{equation}\label{eq:resolution_single_taut}
    \begin{tikzcd}
    \cdots \arrow[r] & \ca{V}_2 \arrow[r] & \ca{V}_1 \arrow[r] & \ca{V}_0 \arrow[r,"\phi"] & \iota_* \left(\bS^{\eta} L_{m-1}^{[d]}\otimes \bS^{\rho} L_{m}^{[d]}\right) \arrow[r] & 0.
\end{tikzcd}
\end{equation}
where 
$
\ca{V}_t:= 
\left(\bS^{\eta}\cB_{1}\boxtimes \bS^{\rho}\cB_{2}\right) \otimes
\bigwedge^t \ca{K}^\vee.
$

We will show below in Lemma~\ref{lem:Vanishing_V^t} that for all $t\ge 1$, all cohomology groups of $\cV_t$ vanish (with the given conditions on $m$ and the partitions involved).
This implies that $\phi$ induces natural isomorphisms of cohomology groups 
\[
H^i\!\left(\Gr(k_1,N_1)\times \Gr(k_2,N_2),\,\bS^{\eta}\cB_1\boxtimes \bS^{\rho}\cB_{2}\right)
\;\cong\;
H^i\!\left(\quot_d,\,\bS^{\eta} L_{m-1}^{[d]}\otimes \bS^{\rho} L_{m}^{[d]}\right)
\]
for all $i\ge 0$.

When $\eta=\gamma$ and $\rho=\lambda$ are partitions, Borel–Weil–Bott theorem (see Remark~\ref{rem:BWB_remark}) tells us that the cohomology groups of $\bS^{\gamma}\cB_1$ and $\bS^{\lambda}\cB_2$ are nonzero only in degree~$0$, where they are naturally isomorphic to $\bS^{\gamma}H^0(V\otimes L_{m-1})$ and $\bS^{\lambda}H^0(V\otimes L_{m})$, respectively. Hence, proving part (ii).

If $\delta$ (resp. $\nu$) is nontrivial, then by the Borel–Weil–Bott theorem (Theorem~\ref{thm:BWB}), all cohomology groups of $\bS^{\eta}\cB_1$ (resp. $\bS^{\rho}\cB_2$) vanish whenever
\[
\delta_1\le k_1 \quad \text{(resp. } \nu_1\le k_2 \text{)}.
\]
Using the assumptions $m\ge d+b$ and $\delta_1+\nu_1<n-r$, we have
\[
\delta_1\le n-r-1\le (n-r-1)(b+d)\le (n-r)m-(b+d)=k_1,
\]
and similarly $\nu_1\le k_2$. Note that $d+b\ge0$ (considering the degree of the subsheaf $S\subset V\subset \cO_{\bP^1}^{\oplus n}$). The extremal case $d+b=0$ implies $S$ is trivial, and thus $b=d=0$, which  corresponds to the case when $\quot_d$ is the Grassmannian; there the claim follows directly from Borel–Weil–Bott.

\end{proof}
\subsection{Technical vanishing lemma}
\begin{lem}\label{lem:Vanishing_V^t}
    For all $t\ge 1$, all cohomology groups of $\cV_t$ in \eqref{eq:resolution_single_taut} vanish, if
    \begin{itemize}[leftmargin=2em]
        \item[(a)] either $\delta$ or $\nu$ is a non-empty partition; or
        \item[(b)] $\delta=\nu=\emptyset$ and $m$ is sufficiently large.
    \end{itemize}
\end{lem}
\begin{proof}
Following the discussion in Section~\ref{sec:Kozsul_resolution}, $\bigwedge^t \ca{K}^\vee$ splits as a direct sum
\[ \bigwedge^t \ca{K}^\vee = \bigoplus_{|\mu|=t} \left(\bS^\mu \cA_1 \boxtimes \bS^\sigma (\cB_2^\vee)^{\oplus c_{\alpha,\beta}^{\mu^\dagger}\cdot c_{\alpha,\beta}^\sigma}\right) \]
where $\mu$ runs over all partitions of size $t$ and $\alpha, \beta$ and $\sigma$ runs over partitions. We see that $\cV_t$ has direct summands of the form
\[  \left(\bS^\mu \cA_1 \otimes \bS^\eta \cB_1 \right) \boxtimes \left(\bS^\chi \cB_2^\vee\right).\]
In order for such a summand to exist and have a nonzero cohomology group, the following conditions must hold simultaneously:
\begin{itemize}[leftmargin=2em]
\item The Littlewood-Richardson coefficients are non-zero, i.e., 
    \[c_{\alpha,\beta}^{\mu^\dagger}\cdot c_{\alpha,\beta}^\sigma\cdot c_{\sigma,-\rho}^{\chi}\ne0.\]
    Here $c_{\sigma,-\rho}^{\chi}$ is the generalized Littlewood-Richardson coefficient \eqref{eq:LR_coeff_general} for $\mathrm{GL}_{r_2}(\bC)$-representaions and $-\rho=(\nu,-\lambda)$. 

 \item Not all cohomology groups of $\bS^\mu \cA_1\otimes \bS^{\eta}\cB_1$ vanish over $\Gr(k_1,N_1)$. Lemma~\ref{lem:BWB_index_multi} implies that we have a well-defined
    $$i:=(r_1;\eta)\text{-index of }\mu.$$
    Recall $\eta=(\gamma,-\delta)$ and note that $\gamma_1^\dagger+\delta_1^\dagger\le r_1$. The index $i$ satisfies the following
      \[\mu_{i + 1 - s} \ge i + r_1 - \gamma_s^\dagger \quad \text{and} \quad \mu_{i + s} \le i + \delta_s^\dagger \quad \text{for all } s \ge 1.\]
      This is illustrated in Figure \ref{fig:young-diagram-mu}.
\begin{figure} 
  \centering
\begin{equation*}
\ytableausetup{boxsize=0.7em}
\begin{tikzpicture}
      \draw (0, 0) node[anchor=north west] (YD) {%
        \ydiagram{20,19,18,17,15, 6, 4, 4,3, 3,2,1}
        *[*(lightgray!70)]{5,5,5,5,5}
      };
\fill[red!50,opacity=0.3] 
    ([xshift=12.2em,yshift=-1.9em]YD.north west)
    rectangle ++(0.75em,-0.75em);
\fill[red!50,opacity=0.3] 
    ([xshift=10.7em,yshift=-2.6em]YD.north west)
    rectangle ++(2.3em,-1.5em);
    
\draw[red, very thick]
  ([xshift=12.9em,yshift=-1.85em]YD.north west) -- 
  ++(0,1.5em);
            \draw[red, very thick]
  ([xshift=12.2em,yshift=-1.85em]YD.north west) -- 
  ++(0.7em,0);
 \draw[red, very thick]
  ([xshift=12.2em,yshift=-1.85em]YD.north west) -- 
  ++(0,-0.7em);
  \draw[red, very thick]
  ([xshift=10.7em,yshift=-2.55em]YD.north west) -- 
  ++(1.5em,0);
  \draw[red, very thick]
  ([xshift=10.7em,yshift=-2.55em]YD.north west) -- 
  ++(0,-1.5em);
   \draw[red, very thick]
  ([xshift=8.5em,yshift=-4.1em]YD.north west) -- 
  ++(2.2em,0);

\draw[black, dotted]([xshift=12.95em,yshift=-1.8em]YD.north west) -- 
  ++(0,-2.3em);
  \draw[black, dotted]([xshift=12.25em,yshift=-2.5em]YD.north west) -- 
  ++(0,-1.6em);

      \fill[blue!50,opacity=0.3] 
    ([xshift=4.1em,yshift=-4.1em]YD.north west)
    rectangle ++(2.8em,-0.7em);
          \fill[blue!50,opacity=0.3] 
    ([xshift=4.1em,yshift=-4.8em]YD.north west)
    rectangle ++(1.4em,-0.7em);
    \draw[blue, very thick]
  ([xshift=7em,yshift=-4.1em]YD.north west) -- 
  ++(1.5em,0em);
            \draw[blue, very thick]
  ([xshift=7em,yshift=-4.1em]YD.north west) -- 
  ++(0,-0.7em);
            \draw[blue, very thick]
  ([xshift=5.5em,yshift=-4.8em]YD.north west) -- 
  ++(1.5em,0em);
  
              \draw[blue, very thick]
  ([xshift=5.5em,yshift=-4.8em]YD.north west) -- 
  ++(0,-0.7em);
          \draw[blue, very thick]
  ([xshift=4.1em,yshift=-5.5em]YD.north west) -- 
  ++(1.4em,0em);
                \draw[blue, very thick]
  ([xshift=4.1em,yshift=-5.5em]YD.north west) -- 
  ++(0,-3.8em);
\draw[black,dotted]([xshift=5.55em,yshift=-4.1em]YD.north west) -- 
  ++(0em,-0.7em);
  \draw[black,dotted]([xshift=6.3em,yshift=-4.1em]YD.north west) -- 
  ++(0em,-0.7em);
\draw[black,dotted]([xshift=4.8em,yshift=-4.1em]YD.north west) -- 
  ++(0em,-1.5em);
  \draw[black,dotted]([xshift=4.1em,yshift=-4.1em]YD.north west) -- 
  ++(0em,-1.5em);

    \draw (4, -3) node {$\mu$};
            \draw (4.95, -1.75) node {\tiny $\textcolor{red}{\gamma_1}$};
         \draw (4.65, -1.75) node {\tiny $\textcolor{red}{\gamma_2}$};
         \draw (4.05, -1.75) node {\tiny $\textcolor{red}\cdots$};
           \draw (3.45, -1.75) node {\tiny $\textcolor{red}{\gamma_q}$};

         \draw (3.2, -1.75) node {\tiny $\textcolor{blue}{\delta_p}$};
         \draw (2.35, -2) node {\tiny $\textcolor{blue}{\cdots}$};
         \draw (2.02, -2.3) node {\tiny $\textcolor{blue}{\delta_2}$};
         \draw (1.72, -2.3) node {\tiny $\textcolor{blue}{\delta_1}$};


    \draw[decorate,decoration={brace,amplitude=3pt},thick]
    ([xshift=0.5em,yshift=0em]YD.north west) -- ++(3.5em,0)
    node[midway,above=4pt] {$i$};
    
    \draw[brown, very thick]([xshift=8.5em,yshift=-4.1em]YD.north west) -- 
  ++(0em,3.8em);
\draw[decorate,decoration={brace,amplitude=3pt},thick]
    ([xshift=4.4em,yshift=0em]YD.north west) -- ++(8.3em,0)
    node[midway,above=4pt] {$r_1$};

\draw[decorate,decoration={brace,mirror,amplitude=3pt},thick]
    ([xshift=-0em,yshift=-0.5em]YD.north west) -- ++(0,-8.7em)
    node[midway,left=4pt] {$k_1$};
    \end{tikzpicture}
\end{equation*}

 \caption{A schematic diagram for the shape of the partition $\mu$, whose $(r_1,\eta)$-index equals $i$. 
 The first $i$ parts of $\mu$ must lie to the right of bold red outline by $\gamma$ and the last $k_1-i$ parts of $\mu$ lie left of blue line defined by $\delta$.}
  \label{fig:young-diagram-mu}
\end{figure}
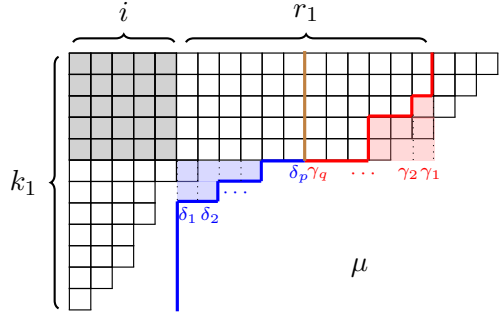

   \item Not all cohomology groups of $\bS^{\chi}\cB_2^{\vee}$ vanish over $\Gr(k_2,N_2)$. In such a case, Lemma~\ref{lem:BWB_index} implies that we have a well-defined
       \[
       j:= k_2\text{-index of }\chi. 
       \]
       The first  inequality in Lemma~\ref{lem:BWB_index} implies that first $j$ entries $\chi_1,\dots\chi_j\ge k_2+j$. The second inequality in Lemma~\ref{lem:BWB_index} and rank of $\cB_2$ implies $\chi_{j+1},\dots,\chi_{r_2}\le j$. This is illustrated  in Figure~\ref{fig:young-diagram-chi}.

\begin{figure} 
  \centering
  \begin{equation*}
  \ytableausetup{boxsize=0.7em}
  \begin{tikzpicture}
        \draw (0, 0) node[anchor=north west] (YD) {%
          \ydiagram{25,23,22,21,20, 5, 5, 4,3, 3,3,2,0,0}
          *[*(lightgray!70)]{5,5,5,5,5}
          *[]{2, 3, 4}
        };
        \draw[red, very thick]
          ([xshift=15.2em,yshift=-0.3em]YD.north west) -- ++(0,-3.8em);
        \draw[blue, very thick]
          ([xshift=4.1em,yshift=-4.1em]YD.north west) -- ++(0,-6.7em);
        \draw[black, dotted]
          ([xshift=-1.7em,yshift=-10.8em]YD.north west) -- ++(5.8em,0em);
        \draw (4, -3) node {$\chi$};
        \draw[decorate,decoration={brace,amplitude=3pt},thick]
          ([xshift=0.2em,yshift=0em]YD.north west) -- ++(3.5em,0)
          node[midway,above=4pt] {$j$};
        \draw[decorate,decoration={brace,amplitude=3pt},thick]
          ([xshift=4.2em,yshift=0em]YD.north west) -- ++(10.5em,0)
          node[midway,above=4pt] {$k_2$};
        \draw[decorate,decoration={brace,mirror,amplitude=3pt},thick]
          ([xshift=-0em,yshift=-0.5em]YD.north west) -- ++(0,-10em)
          node[midway,left=4pt] {$r_2$};
  \end{tikzpicture}
  \end{equation*}

  \caption{A schematic diagram for the shape of $\chi$ when it is a partition: The $k_2$-index equals $j$. The first $j$ entries of $\chi$ must lie to the right of bold red line and the last $r_2-j$ entries lie left of blue line.}\label{fig:young-diagram-chi} 
\end{figure}
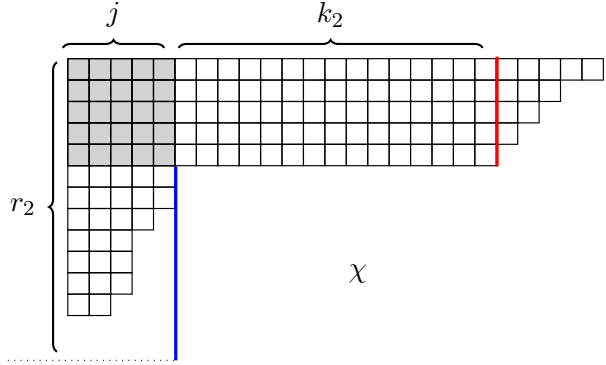

 \end{itemize}
We will apply the properties of Littlewood--Richardson coefficients, listed in Proposition~\ref{prop:LR_coefficients_properties}, to arrive at a contradiction on the existence of the indices $j$ and $i$, which would then show that the conditions above cannot simultaneously be satisfied, whence all the cohomology groups of $\cV_t$ must be zero. We divide this into four exhaustive cases:\
\\
\\
\textbf{Case 1:} Suppose $j > i$. 
By Weyl’s inequality~\eqref{eq:WeylInequality}, the condition \(c_{\alpha,\beta}^{\sigma}\neq 0\) implies
\[
\alpha_j+\beta_1\ge \sigma_j.
\]
Similarly, from \(c_{\sigma,-\rho}^{\chi}\neq 0\) we obtain (recall $-\rho=(\nu,-\lambda)$)
\[
\sigma_j+\nu_1\ge \chi_j \ge j+k_2.
\]
Note that \(\beta_1\le \mu^\dagger_1\le k_1\) (since \(c_{\alpha^\dagger,\beta^\dagger}^{\mu}\neq 0\)) and \(k_2-k_1=n-r\).
Combining the two inequalities above yields,
\[
\alpha_j \ge \sigma_j - \beta_1 \ge j + k_2 - \nu_1 - k_1 \ge j + n-r - \nu_1.
\]
Hence the conjugate of $\alpha$ satisfy
\begin{align}\label{eq:j>alphadagger}
    j\le \alpha^\dagger_{j+(n-r-\nu_{1})}. 
\end{align}
Using $c_{\alpha^\dagger,\beta^\dagger}^{\mu}\ne0$ and the assumption $j>i$, we have
\begin{align}\label{eq:j>i-alphadagger}
    \alpha^\dagger_{j+(n-r-\nu_{1})}\le \mu_{j+(n-r-\nu_{1})}\le \mu_{i+(n-r-\nu_{1})}\le i+\delta^{\dagger}_{(n-r-\nu_{1})}
\end{align}
The assumption $\delta_1+\nu_{1}\le n-r-1$ implies that $\delta^{\dagger}_{(n-r-\nu_{1})}=0$. Hence we obtain $j\le i$ from \eqref{eq:j>alphadagger} and \eqref{eq:j>i-alphadagger}, which is a contradiction.\
\\
\\
\textbf{Case 2:} Suppose $j < i\le 2j$. Applying dominance inequalities in Proposition~\ref{prop:LR_coefficients_properties} for
\[
\sigma \le \alpha + \beta \quad \text{and} \quad \alpha \cup \beta \le \mu^\dagger,
\]
arising from $c_{\alpha,\beta}^\sigma$ and $c_{\alpha,\beta}^{\mu^\dagger}$, respectively, we obtain:
\begin{equation}
        \begin{aligned}[b]
   \sigma_1 + \cdots + \sigma_j 
               &\le \alpha_1 + \cdots + \alpha_j + \beta_1 + \cdots + \beta_j \\
               &\le \mu_1^{\dagger} + \cdots + \mu_{2j}^{\dagger} \\
               &\le ik_1 + (2j - i)i + |\delta|. 
        \end{aligned}
\label{eq:first_dominance}
\end{equation}
The last inequality is evident by analyzing the shape of $\mu$ in Figure~\ref{fig:young-diagram-mu}. Similarly, using the dominance inequalities for 
\[
\mu \le \alpha^\dagger + \beta^\dagger \quad \text{and} \quad \alpha^\dagger \cup \beta^\dagger \le \sigma^\dagger,
\]
and analyzing Figure~\ref{fig:young-diagram-mu} we obtain:
\begin{align}
\begin{aligned}[b]
    i(r_1 + i) - |\gamma| &\le \mu_1 + \cdots + \mu_i \\
                                &\le \alpha^\dagger_1 + \cdots + \alpha^\dagger_i + \beta^\dagger_1 + \cdots + \beta^\dagger_i \\
                                &\le \sigma_1^{\dagger} + \cdots + \sigma_{2i}^{\dagger} 
\end{aligned}\label{eq:second_dominance}
\end{align}
The non-vanishing of generalized Littlewood-Richardson coefficient $c_{\sigma,-\rho}^{\chi}\ne 0$ implies the following two identities:
\begin{align}
    \sigma_1 + \cdots + \sigma_j &\ge \chi_1 + \cdots + \chi_j-|\nu|\label{eq:sigma} \\
    \sigma_1^{\dagger} + \cdots + \sigma_{2i}^{\dagger} &\le\chi_1^{\dagger} + \cdots + \chi_{2i}^{\dagger} +|\lambda| \label{eq:sigmadagger}
\end{align}
The first inequality \eqref{eq:sigma} simply follows from the dominance inequality \eqref{eq:dominance_inequality}, while the second inequality \eqref{eq:sigmadagger} requires an argument, that we provide in Lemma~\ref{lem:LR_dagger}. For the tuple $\chi$, we define $\chi_s^\dagger$ for the number of entries in $\chi$ that are at least $s$.  

Combining the inequalities \eqref{eq:first_dominance} and \eqref{eq:sigma}, and analyzing the shape of $\chi$ in Figure~\ref{fig:young-diagram-chi}, we obtain
\begin{align}\label{eq:chi_ineq}
    j(k_2+j)-|\nu|\le \chi_1 + \cdots + \chi_j-|\nu|\le ik_1 + (2j - i)i + |\delta|.
\end{align}
Similarly, by combining \eqref{eq:second_dominance} and \eqref{eq:sigmadagger} and analyzing Figure~\ref{fig:young-diagram-chi}, we have 
\begin{align}\label{eq:chidagger_ineq}
     i(r_1+ i) - |\gamma|\le\chi_1^{\dagger} + \cdots + \chi_{2i}^{\dagger} +|\lambda|\le  jr_2+(2i-j)j+|\lambda|
\end{align}
These two inequalities \eqref{eq:chi_ineq} and \eqref{eq:chidagger_ineq} can be rearranged into the following useful form:
\begin{align}
    j(k_2 - k_1) &\le (i - j)k_1 - (i - j)^2 + |\delta|+|\nu|,
    \\
    j(r_2-r_1) &\ge (i - j)r_1 + (i - j)^2 - |\gamma|-|\lambda|.\label{eq:j(r_2-r_1)}
\end{align}
Recall that $k_2 - k_1 = n - r$ and $r_2-r_1 = r$. Multiplying the first inequality by $r$ and the second by $n - r$, we get:
\begin{align*}
    r(|\delta|+|\nu|) + (n - r)(|\gamma|+|\lambda|) &\ge (i - j)\left((n - r)r_1 - rk_1\right) + n(i - j)^2 \\
                                &= (nd+rb)(i - j) + n(i - j)^2 \\
                                &\ge nd+rb+n,
\end{align*}
where we used the inequality $i - j \ge 1$ in the final step.
This contradicts the hypothesis of the lemma.
\
\\
\\
\textbf{Case 3:} Suppose $2j < i$. In the previous case, the assumption $i\le 2j$ was used only in \eqref{eq:chi_ineq}, so using our new assumption $2j < i$, we can replace \eqref{eq:chi_ineq} by 
\begin{align}\label{eq:chidagger_2j<i}
    j(k_2+j)-|\nu|\le ik_1,
\end{align}
while \eqref{eq:chidagger_ineq} stays the same. We can rewrite \eqref{eq:j(r_2-r_1)} and \eqref{eq:chidagger_2j<i} as
\begin{align}
    j(n-r)=j(k_2 - k_1) &\le (i - j)k_1 - j^2 +|\nu|,\label{eq:Case3_1}
    \\
    jr=j(r_2-r_1) &\ge (i - j)r_1 + (i - j)^2 - |\gamma|-|\lambda|.\label{eq:Case3_2}
\end{align}
Multiplying \eqref{eq:Case3_1} and \eqref{eq:Case3_2} by $r$ and $n-r$ respectively, we obtain
\begin{align}
\begin{aligned}[b]
    r|\nu| + (n-r)(|\gamma|+|\lambda|) &
    \geq 
    (i-j)((n-r)r_1 - rk_1) + rj^2 + (i-j)^2(n-r) 
    \\ 
    & > (nd +rb)(i-j) + rj^2 + (n-r)(i-j)^2\label{eq:i>2j_ineq}
    \end{aligned}
\end{align}
When $j = 0$, we have $\chi_1 \le 0$. By Littewood-Richardson rule, we have $\nu_1 \le \chi_1$, so $\lvert \nu \rvert = 0$. Let us assume $j > 0$ or $|\delta|>0$, i.e. the hypothesis of part (a). Then using \eqref{eq:i>2j_ineq}, we obtain the inequality
$$r(|\delta|+|\nu|) + (n-r)(|\gamma| + |\lambda|) \geq nd+rb+n,$$
which is a contradiction.

So let us suppose that $j = 0$ and $|\delta| = 0$, i.e. the hypothesis of part (b). Using $j=0$, \eqref{eq:Case3_2} gives us $|\lambda|+|\gamma|\ge ir_1+i^2\ge r_1=rm+d$, so we get a contradiction for $m$ sufficiently large. \
\\
\\
\textbf{Case 4:} Suppose $i = j$. Since $|\mu|=t\ge 1$, $i\ne 0$. We can slightly modify the inequality \eqref{eq:chi_ineq} as follows:
\begin{align*} 
    j(k_2+j-\nu_1)
    & \leq \sigma_1 + \cdots + \sigma_j \\ 
    & \leq \mu_1^{\dagger} + \cdots + \mu_{2j}^{\dagger}  \\  
    & \le ik_1 + i(2j-i+\delta_{1}) .
\end{align*}
 But since $i = j$, we get $\nu_{1} + \delta_{1} \geq n-r+i\ge n-r$, which is a contradiction.
\end{proof}

\begin{lem}\label{lem:LR_dagger}
    Suppose the generalized Littlewood--Richardson coefficient $c_{\sigma,(\nu.-\lambda)}^{\chi}\ne 0$, for $\mathrm{GL}_{r_2}(\bC)$ representations, where $\sigma $ is a partition. Then for any positive integer $s$,
    \[
     \sigma_1^{\dagger} + \cdots + \sigma_{s}^{\dagger} -|\lambda| \le\chi_1^{\dagger} + \cdots + \chi_{s}^{\dagger} ,
    \]
\end{lem}
\begin{proof}
Let $\chi=(\chi_1,\dots,\chi_{r_2})=(\omega,-\theta)$, and note that $\chi^\dagger_j=\omega^\dagger_j$ for all $j>0$. The size constraint implies the equality 
\begin{equation}
\label{eqn: size_constraint_eq}
  |\sigma|+|\nu|-|\lambda|=|\omega|-|\theta|.  
\end{equation}
It suffices to prove the inequality
\[
\sigma^\dagger_{s+1}+\cdots+\sigma^\dagger_M+|\nu|\ge \chi^\dagger_{s+1}+\cdots+\chi^\dagger_{M}\ge \chi^\dagger_{s+1}+\cdots+\chi^\dagger_{M} -|\theta|,
\]
and the the result follows after subtracting it from equation~\eqref{eqn: size_constraint_eq}.

Fix an integer $\ell\ge \max\{\theta_1,\lambda_1\}$, then we have the classical Littlewood-Richardson coefficient $c_{\sigma,\tau}^{\phi}$, where $\tau=(\nu,-\lambda)+(\ell)^{r_2}$ and $\phi=\chi+(\ell)^{r_2}$ are partitions. Apply the Horn's inequality \eqref{eq:horns} coming from $c_{\sigma^\dagger,\tau^\dagger}^{\phi^\dagger}\ne 0$, we obtain
\[
\sigma^\dagger_{s+1}+\cdots+\sigma^\dagger_{M} +\tau^\dagger_{\ell+1}+\cdots+\tau^\dagger_{\ell+M} \ge \phi^\dagger_{\ell+s+1}+\cdots+\phi^\dagger_{\ell+M}
\]
Note that $\tau^\dagger_{\ell+j}=\nu_j^\dagger$ and $\phi^\dagger_{\ell+s+j}=\chi^\dagger_{s+j}$ for all $j>0$. Hence we finish the proof by letting $M$ large enough.
\end{proof}

\subsection{Extension groups}
We will now specialize Theorem~\ref{thm:two_insertions} to calculate the extension groups and prove all the conclusions in Theorem~\ref{thm:Ext_implications}. Note that the cohomology groups of Schur functors of two dual tautological bundles in Theorem~\ref{thm:dual_two_insertion_intro} immediately follow from Theorem~\ref{thm:two_insertions}.

\begin{proposition}\label{prop:Ext_groups_single}
    Assume $V$ splits as a direct sum of line bundles of nonpositive degrees. Let $L$ be a line bundle with degree at least $d+b$. Let $\nu$ and  $\lambda$ be two partitions such that $r\lvert \nu \rvert + (n-r)\lvert \lambda \rvert < nd+rb+n$ and $\nu_1 <n-r$. Then
    \begin{align*}
        \Ext^{i}\left(\bS^\nu L^{[d]}, \bS^\lambda L^{[d]} \right) = \begin{cases}
          \Hom\left(\bS^\nu \cB, \bS^\lambda \cB \right)&i=0\\
        0 &i\ge1.
            \end{cases}
    \end{align*}
        where $\cB$ is the universal quotient bundle on the Grassmannian $\Gr(k,N)$, where $N = \dim H^0(V\otimes L) $ and $k=N-\rank L^{[d]}$.
\end{proposition}
\begin{proof}
We follow the notation in the proof of Theorem \ref{thm:two_insertions}. We may express  
\[
\bS^{\nu}(L^{[d]})^\vee \otimes \bS^{\lambda} L^{[d]} 
= \bigoplus_{\eta} \left(\bS^{\eta}(L^{[d]})^{\vee}\right)^{\oplus c_{\nu,-\lambda}^{\eta}},
\]
where $\eta=(\eta_1,\eta_2,\dots,\eta_{N_1-k_1})=(\gamma,-\delta)$ runs over tuples 
of non-increasing integers, and $c_{\nu,-\lambda}^{\eta}$ are the modified 
Littlewood--Richardson coefficients \eqref{eq:modified_LR_coeff}. Horn’s inequalities imply that  
\[
\eta_j \le \nu_j + (-\lambda)_1 \le \nu_j \quad \text{for all } j \ge 1,
\]  
hence $|\gamma|\le |\nu| $ and $\eta_1 \le \nu_1 < n-r$. Similarly, we obtain $|\delta|\le |\lambda|$, and thus 
Theorem~\ref{thm:two_insertions} implies that all higher cohomology groups vanish and the zeroth cohomology group is expressed in terms of the cohomology groups of $\cB_1$ as stated.
\end{proof}
\begin{proof}[Proof of Theorem~\ref{thm:Ext_implications}]
Proposition~\ref{prop:Ext_groups_single} immediately implies part~(i), and a straightforward application of the Borel--Weil--Bott theorem yields part~(ii). The statement about the exceptional collection follows by noting that $\lambda_1 < n-r$ and $m \ge b+d$ imply $\lambda_1 \le k_1$, and by invoking \cite{Kapranov} to observe that
\[
\{\bS^{\lambda}\cB_1 \;:\; |\lambda|\le d,\ \lambda_1 < n-r\}
\]
is contained in a full exceptional collection for $\Gr(k_1, N_1)$.
\end{proof}

\subsection{Cohomological degree of the tautological bundles}
In Section~\ref{sec:Kozsul_resolution}, we showed that, except for the space of global sections of 
$\bS^{\eta} L_{m-1}^{[d]} \otimes \bS^{\rho} L_{m}^{[d]}$, all cohomology groups vanish, 
and we explicitly computed $H^{0}$ under size constraints on $\eta$ and $\rho$. 
Here we remove the constraints on $\eta$ and $\rho$ and establish the vanishing of the higher cohomology 
groups, strengthening \cite[Theorem~1.3]{Lin2025}.
\begin{proposition}\label{prop:cohomological_degree_no_size_constraints}
    Assume $V$ splits as a direct sum of line bundles of nonpositive degrees. Let $\eta = (\gamma, -\delta)$, $\rho = (\lambda, -\nu)$. Then
    for all  $m \geq d+b$,
    \[
    H^{D}(\quot_d,\bS^{\eta} L_{m-1}^{[d]} \otimes \bS^{\rho} L_{m}^{[d]}) = 0 \quad \text{for all } D>|\delta|+|\nu|.
    \]
    In particular, when $\eta = \gamma$ and $\rho = \lambda$ are partitions, 
    \[
    H^{D}(\quot_d,\bS^{\gamma} L_{m-1}^{[d]} \otimes \bS^{\lambda} L_{m}^{[d]}) = 0\quad \text{for all }D>0.
    \]
\end{proposition}
\begin{proof}
  First, recall the Koszul resolution \eqref{eq:resolution_single_taut}, and note that it suffices to show
\[
H^D\!\big(\Gr(k_1,n_1)\times \Gr(k_2,n_2),\cV_t\big)=0 \quad \text{for all } t\ge0,\text{ and} \ D> t+|\delta|+|\nu|.
\]
Indeed, assume there exist partitions $\mu,\alpha,\beta,\sigma,\chi$ such that
\[
H^D\!\Big(\Gr(k_1,n_1)\times \Gr(k_2,n_2),
\big(\bS^\mu \cA_1 \otimes \bS^\eta \cB_1 \big)\boxtimes \big(\bS^\chi \cB_2^\vee\big)\Big)\neq 0,
\]
with $c_{\alpha,\beta}^{\mu^\dagger}\cdot c_{\alpha,\beta}^{\sigma}\cdot c_{\sigma,-\rho}^{\chi}\neq 0$. Recall the indices $i$ and $j$ from the proof of Lemma~\ref{lem:Vanishing_V^t}. Analyzing the non-vanishing of the cohomology groups and the Littlewood–Richardson coefficients,
in particular, the inequalities \eqref{eq:first_dominance}, \eqref{eq:sigma}, and \eqref{eq:chi_ineq}, implies
\[
j(k_2+j)-|\nu| \;\le\; \mu_1^{\dagger}+\cdots+\mu_{2j}^{\dagger}.
\]
By Lemma~\ref{lem:BWB_index_multi}(ii) the cohomological degree satisfies
\[
D \le k_2 j + |\delta|+ (\mu_1+\dots+\mu_i)-i^2= k_2 j + |\delta|+ i p  + \sum_{s\ge i+p+1}\mu_s^\dagger,
\]
where $k_2 j$ is the cohomological degree contributed by $\bS^\chi \cB_2^\vee$. Adding the two inequalities above yields
\[
|\delta|+|\nu|+|\mu|
\;\ge\;
D+\begin{cases}
j^2-ip+\displaystyle\sum_{s=2j+1}^{i+p}\mu_s^\dagger,& \text{if } 2j\le i+p,\\[0.4em]
j^2-ip,& \text{if } 2j> i+p.
\end{cases}
\]
In both cases on the right, the expression is nonnegative, which is precisely what we need.
Indeed, if $2j> i+p$, then $j^2-ip\ge (i+p)^2/4 - ip\ge 0$, and if $2j\le i+p$, then
\[
j^2-ip+\sum_{s=2j+1}^{i+p}\mu_s^\dagger
\;\ge\;
j^2-ip+i(i+p-2j)
=
(i-j)^2
\;\ge\; 0.
\]

\end{proof}

\section{Multiple line bundles and vector bundles}
In the previous section, we calculated the cohomology groups of $\bS^{\eta}L_{m-1}^{[d]} \otimes \bS^{\rho} L_m^{[d]}$ for sufficiently positive $m$. We required $m$ to be large enough so that these tautological bundles are pulled back along a Str\o mme embedding. In order to access those degrees inaccessible through a Str\o mme embedding, we work with Schur complexes, introduced in Subsection \ref{sec:Schur_complex} below. This will allow us to bootstrap Theorem \ref{thm:two_insertions} to prove Theorem~\ref{thm:intro_vb_insertion} in full generality. 
\subsection{Schur functor of perfect complexes}\label{sec:Schur_complex}
For a partition $\lambda$ of $n$, let $\pi_\lambda$ be the irreducible representation of $S_n$
corresponding to $\lambda$. Then Schur functors can be defined for a perfect complex $A$ of arbitrary length as $\bS^{\lambda}(A) = (\pi_\lambda \otimes A^{\otimes n})^{S_n}$ (see for instance Page 302 between Theorems 9.11.4 and 9.11.5 in \cite{etingof2017tensor}). From this definition, it is also clear that in characteristic zero, Schur functors respect quasi-isomorphisms: taking tensor products respect quasi-isomorphisms, and taking $S_n$ invariants is exact in characteristic zero. 

When we write $\bS^{\lambda} M^{[d]}$ for $M^{[d]}$ is an object in the derived category of the Quot scheme, we mean to represent $M^{[d]}$ first as a perfect complex and take the Schur functor. In our case, we can actually always represent $M^{[d]}$ and $M^{\{d\}}$ as a map of vector bundles. Below we recall definitions and basic properties of Schur functors associated to a partition and a map of vector bundles. The construction in this case is called Schur complexes in the literature. See \cite[Chapter 2]{weyman} for details.

Let $X$ be a locally Noetherian scheme and let $$\Phi: E_1 \to E_2$$ be a map of vector bundles on $X$.  

\begin{definition}\label{def:Schur_complex}
Let $\lambda$ be a partition. We write  $\bS^{\lambda}(\Phi)^{\bullet}$ for the Schur complex of length $|\lambda|$, so that 
\begin{equation*}
    \bS^{\lambda}(\Phi)^{q} = \bigoplus_{\nu \subseteq \lambda, |\nu| = q} \bS^{\lambda/\nu} E_1 \otimes \bS^{\nu^{\dagger}} E_2. 
\end{equation*}
is placed in $q$-th degree.
The beginning of $\bS^{\lambda}(\Phi)^{\bullet}$ looks like
\[
\bS^{\lambda}E_1 \to \bS^{\lambda / (1)} E_1 \otimes E_2 \to 
\bS^{\lambda/(2)} E_1 \otimes \wedge^2 E_2 
\oplus
\bS^{\lambda/(1^2)} E_1 \otimes \Sym^2 E_2 
\to \cdots 
\]
Sometimes it is useful to index homologically as well. 
We write $\bS^{\lambda}(\Phi)_{\bullet}$ for the Schur complex of length $|\lambda|$ associated to the map $\Phi$, so that
\begin{equation*}
    \bS^{\lambda}(\Phi)_{q} = \bigoplus_{\nu \subseteq \lambda, |\nu| = q} \bS^{\nu^{\dagger}} E_1 \otimes \bS^{\lambda/\nu} E_2. 
\end{equation*}
is placed in $(-q)$-th degree.
The beginning of $\bS^{\lambda}(\Phi)_{\bullet}$ looks like
\[
\cdots \to 
\Sym^2 E_1 \otimes \bS^{\lambda / (1^2)} E_2 
\oplus 
\wedge^2 E_1 \otimes \bS^{\lambda / (2)} E_2
\to E_1 \otimes \bS^{\lambda / (1)} E_2 \to \bS^{\lambda} E_2.
\]
Note that $\bS^{\lambda}(\Phi)_{\bullet} = \bS^{\lambda^{\dagger}}(\Phi)^{\bullet}[\ell]$, where $\ell= |\lambda|$.
\end{definition}

\begin{proposition} \label{prop:SchurExact}
    If we have a short exact sequence
    \begin{equation}
        0 \to E_1 \xrightarrow{\Psi} E_2 \xrightarrow{\Phi} E_3 \to 0, \label{eq:SES_prop_SchurComplex}        
    \end{equation}
    where each $E_i$ is a vector bundle,
    then $\bS^{\lambda}(\Psi)_{\bullet} \to \bS^{\lambda} E_3 \to 0$ is exact, and $0 \to \bS^{\lambda}E_1 \to \bS^{\lambda}(\Phi)^{\bullet}$ is exact.
\end{proposition}
\begin{proof}
    Since $E_3$ is locally free, the short exact sequence \eqref{eq:SES_prop_SchurComplex} is locally split. Acyclicity can be checked locally, so \cite[Corollary V.1.15]{KaanBuchsbaumWyman} implies that $(\bS^\lambda \Psi)_{\bullet}$ is acyclic. The same result also shows that the cokernel locally agrees with $\bS^{\lambda} E_3$, and the functoriality of Schur complexes shows that the isomorphisms must glue, so, globally, the cokernel is also exactly $\bS^{\lambda} E_3$.

    The first statement applied to the dual of \eqref{eq:SES_prop_SchurComplex} shows that the complex $S^{\lambda}(\Phi^{\vee}: E_3^{\vee} \to E_2^{\vee}) \to S^{\lambda} E_1^{\vee} \to 0$ is exact. Dualizing this complex shows the second statement. 
\end{proof} 

\begin{rem}
In the theorem statements, such as Theorem \ref{thm:intro_vb_insertion}, we sometimes write $\bS^{\lambda}H^\bullet(V \otimes M)$, where $V, M$ are vector bundles on $\bP^1$. The definition of the Schur functor is rather trivial in this case, but requires caution. For instance, $\bS^{\lambda}H^{\bullet}(\cO(-k))$, where $k > 1$, is the vector space $\bS^{\lambda^{\dagger}} H^1(\cO(-k))$, concentrated in degree $|\lambda|$. 
\end{rem}

\subsection{Vanishing results for two insertions}
In Theorem~\ref{thm:two_insertions}, we computed the cohomology groups of Schur functors of $L_{m-1}^{[d]}$ and $L_{m}^{[d]}$. We shall now use the properties of Schur complex to conclude results about vector bundles $L_{m-1}^{\{d\}}$ and $L_{m}^{\{d\}}$ and their Schur functors.
\begin{lemma}\label{lem:two_L^{m}}
Let  $L_{m-1}$ and $L_{m}$ line bundles with $m$ sufficiently large. Let $\mu,\nu,\gamma,\lambda$ be partitions satisfying $$|\mu|+|\nu|+|\gamma|+|\lambda|<(nd+rb+n)/(n-r).$$ Then all cohomology groups of the vector bundles 
    \[
    \bS^{\mu}L_{m-1}^{\{d\}}\otimes \bS^{\nu}L_{m}^{\{d\}} \otimes \bS^{\gamma}L_{m-1}^{[d]}\otimes \bS^{\lambda}L_{m}^{[d]}
    \]
    vanish unless $\mu=\nu=\emptyset$.
\end{lemma}
\begin{proof}
    Let $m$ be sufficiently large such that part (ii) of Theorem~\ref{thm:two_insertions} holds. Now consider the exact sequence of vector bundles
    \[
    \begin{tikzcd}
    0 \arrow[r] &L_{m}^{\{d\}} \arrow[r] & W_m \arrow[r,"\Phi_m"] & L_m^{[d]} \arrow[r] & 0 
    \end{tikzcd}
    \]
    where $W_{m}:= \pi_*(p^*(V\otimes L_m))$. 
    
    We will first prove the case $\mu=\emptyset$ and $\nu\ne \emptyset$. Applying the Schur functor $\bS^\nu$ to the above, then tensoring with the vector bundles $F:=\bS^{\gamma}L_{m-1}^{[d]}\otimes \bS^{\lambda}L_{m}^{[d]} $ we obtain a complex
    \[
    \begin{tikzcd}
    0 \arrow[r] &F\otimes \bS^{\nu}L_{m}^{\{d\}} \arrow[r] &F\otimes \bS^{\nu}(\Phi_m)^\bullet 
    \end{tikzcd}
    \]
    which is exact by applying Proposition~\ref{prop:SchurExact}. Therefore, the cohomology groups of the vector bundle $F\otimes \bS^{\nu}L_{m}^{\{d\}}$ equals the hypercohomology groups of the complex $F\otimes \bS^{\nu}(\Phi_m)^\bullet$. Now consider the hypercohomolgy spectral sequence
    \begin{equation}\label{eq:spectral_sequence}
          E_{1}^{p,q}\;\Rightarrow\;H^{p+q}(\quot_d,F\otimes \bS^{\nu}(\Phi_m)^\bullet)
    \end{equation}
    whose first page is given by
    \[
    E_{1}^{p,q} :=H^q\left(\quot_d, F\otimes \bS^{\nu}(\Phi_m)^p\right).
    \]
    By the definition of Schur complex, we note that $F\otimes \bS^{\nu}(\Phi_m)^p$ splits as a direct sum with summands of the form $F\otimes \bS^{\alpha}{W_{m}}\otimes \bS^{\beta}L_m^{[d]}$, such that $|\alpha|+|\beta|=|\nu| $ and $|\beta|=p$. Note that ${W_{m}}$ is a trivial bundle. Applying Littlewood-Richardson rule, we further split $F\otimes \bS^{\alpha}{W_{m}}\otimes \bS^{\beta}L_m^{[d]}$ into direct sum, with summands of form 
    \begin{equation}\label{eq:summand_m-1,m}
        \bS^{\gamma}L_{m-1}^{[d]}\otimes \bS^{\eta}L_{m}^{[d]}
    \end{equation}
     where  $$|\gamma|+|\eta|= |\gamma|+|\lambda|+|\beta|\le|\gamma|+|\lambda|+|\nu| <(nd+rb+n)/(n-r).$$ Theorem~\ref{thm:two_insertions} implies that $q$-th cohomology groups of each of these summand vanish for all $q\ge1$. Therefore, $H^q(\quot_d,F\otimes \bS^{\alpha}{W_{m}}\otimes \bS^{\beta}L_m^{[d]})=0$, and thus $E_1^{p,q}=0$ for all $q\ge1$. We will now show that the first row corresponding to $q=0$ 
    \begin{equation}\label{eq:spectral_Es,0}
    0\to E_1^{0,0}\to E_{1}^{1,0}\to \cdots\to E_1^{|\nu|,0}\to 0
    \end{equation}
    in page one is exact. Consider the diagram
    \[
    \begin{tikzcd}
    W_m \arrow[r, "id"] \arrow[d] & W_m  \arrow[d]  \\
    W_m \arrow[r, "\Phi_m"] & L_m^{[d]} 
    \end{tikzcd}
    \]
    and the induced morphisms of complexes
 \[   \begin{tikzcd}
     \Tilde{F}\otimes \bS^{\nu}(id)^\bullet\arrow[r]&F\otimes \bS^{\nu}(id)^\bullet \arrow[r]& F\otimes \bS^{\nu}(\Phi_m)^\bullet .
    \end{tikzcd}
    \]
    where $\Tilde{F}=\bS^{\gamma}W_{m-1}\otimes \bS^{\lambda}W_{m}$.
Theorem~\ref{thm:two_insertions} implies that the induced morphism for each summand \eqref{eq:summand_m-1,m} in the $p$-th degree
\[
\bS^{\gamma}W_{m-1}\otimes \bS^{\eta }{W_{m}}\to \bS^{\gamma}L_{m-1}^{[d]}\otimes \bS^{\eta}L_m^{[d]}
\]
induces isomorphisms for the space of global sections, and hence induces isomorphisms  
\[
H^0(\quot_d,\Tilde{F}\otimes \bS^{\nu}(id)^p)\cong H^0(\quot_d,F\otimes \bS^{\nu}(\Phi_m)^p)=E_1^{p,0}.
\]
Note that  the complex of trivial bundles $\Tilde{F}\otimes \bS^{\nu}(id)^\bullet$ (equivalently of their space of global sections) is exact by Proposition~\ref{prop:SchurExact}, and hence \eqref{eq:spectral_Es,0} is exact. Therefore the spectral sequence \eqref{eq:spectral_sequence} degenetates in page two with all the terms zero, so we conclude that 
$$
 H^{D}(\quot_d,\bS^{\nu}L_{m}^{\{d\}} \otimes \bS^{\gamma}L_{m-1}^{[d]}\otimes \bS^{\lambda}L_{m}^{[d]}) \cong H^{D}(\quot_d,F\otimes \bS^{\nu}(\Phi_m)^\bullet) =0 $$ 
 for all $D\ge 0$. The same proof works when $\nu=\emptyset$ and $\mu\ne \emptyset$, and it gives us 
 \begin{equation}\label{eq:nu=zero _vanish} 
 H^{D}(\quot_d,\bS^{\mu}L_{m-1}^{\{d\}} \otimes \bS^{\gamma}L_{m-1}^{[d]}\otimes \bS^{\lambda}L_{m}^{[d]}) =0 \quad \text{for all }D\ge0.
\end{equation}

When both $\mu$ and $\nu$ are non-trivial, we again consider the exact sequence, 
 \[
    \begin{tikzcd}
    0 \arrow[r] &F\otimes \bS^{\nu}L_{m}^{\{d\}} \arrow[r] &F\otimes \bS^{\nu}(\Phi_m)^\bullet 
    \end{tikzcd}
    \]
    where $F:=\bS^{\mu}L_{m-1}^{\{d\}}\otimes\bS^{\gamma}L_{m-1}^{[d]}\otimes \bS^{\lambda}L_{m}^{[d]} $. Running the same argument as above, we note that, for each $p$, the bundle $F\otimes \bS^{\nu}(\Phi_m)^p$ splits as a direct sum with summands of the form 
    $$\bS^{\mu}L_{m-1}^{\{d\}}\otimes \bS^{\gamma}L_{m-1}^{[d]}\otimes \bS^{\eta}L_{m}^{[d]} $$
    for partitions $\eta$ with $|\eta|\le |\lambda|+|\nu|$.
    Since $\mu$ is nontrivial, all cohomology of the above summands vanish by \eqref{eq:nu=zero _vanish}, and hence
    $$H^{q}(\quot_d,F\otimes \bS^{\nu}(\Phi_m)^p)=0$$ 
    for all $p$ and $q$. This implies that each object in the first page of the corresponding hypercohomology spectral sequence is zero, and thus $H^{D}(\quot_d,F\otimes \bS^{\nu}L_{m}^{\{d\}})=0$ for all $D\ge0$.
\end{proof}
We now prove Theorem \ref{cor:vanishing_E_x} about the vanishing of the cohomology of $\cS_x$, restriction of the universal subsheaf $\cS$ to $\quot_d\times \{x\}$, using Lemma \ref{lem:two_L^{m}}.
\begin{proof}[Proof of Theorem~\ref{cor:vanishing_E_x}]
Let $L_{m}=\cO(m)$ and $L_{m-1}=\cO(m-1)$ for the same choice of $m$ as in part (ii) of Theorem~~\ref{thm:two_insertions} and Lemma~\ref{lem:two_L^{m}}. For a point $x\in \bP^1$, the short exact sequence
\[
0\to L_{m-1}\to L_m\to \cO_x\to 0
\]
induces the short exact sequence involving tautological bundles
 \[
    \begin{tikzcd}
    0 \arrow[r] &L_{m-1}^{\{d\}} \arrow[r,"\Psi"] &L_{m}^{\{d\}}\arrow[r] &\cE_x \arrow[r]&0.
    \end{tikzcd}
    \]
on $\quot_d$.
Proposition~\ref{prop:SchurExact} provides a resolution
\[\begin{tikzcd}
    \bS^{\lambda}(\Psi)_{\bullet} \arrow[r]& \bS^{\lambda} \cE_x \arrow[r] & 0
\end{tikzcd}\]
for a partition $\lambda$. When $0\ne|\lambda|<(nd+rb+n)/(n-r)$, Lemma~\ref{lem:two_L^{m}} applies to each term in the resolution, and hence proving that all cohomology groups of  $\bS^{\lambda} \cE_x$ vanish. The statement involving multiple partitions is a simple consequence of Littlewood-Richardson rule.
\end{proof}

\subsection{Proof of Theorem~\ref{thm:intro_vb_insertion}}
As noted in Remark~\ref{rem:vb_to_multi_line}, every vector bundle on \(\bP^1\) splits as a direct sum of line bundles, so Theorem~\ref{thm:intro_vb_insertion} follows from the following statement concerning multiple line bundles.
\begin{theorem}\label{thm:multiple_sub_quot}
Let $K_1,K_2,\dots K_s$ and $M_1,M_2,\dots,M_{t}$ be line bundles on $\bP^1$. For any tuples of partitions $\mu^1,\mu^2,\dots,\mu^{s}$ and $\lambda^1, \lambda^2,\ldots, \lambda^t$ satisfying
$$ \lvert \mu^1\rvert+\cdots+|\mu^s|+|\lambda^1|+\cdots+|\lambda^t| < (nd+rb+n)/(n-r),$$
\begin{itemize}[leftmargin=2em]
    \item[(a)] If $\mu^i$ is a nontrivial partition for some $1\le i\le s$, then 
    \begin{align*}
         H^\bullet\left(\quot_d, \bigotimes_{i=1}^{s} \bS^{\mu^i} K_i^{\{d\}}\otimes \bigotimes_{j=1}^{t}\bS^{\lambda^j} M_j^{[d]} \right) =  
        0.
    \end{align*} 
    \item[(b)] If $\mu^1=\cdots=\mu^{s}=\emptyset$, then 
\[
H^\bullet\left(\quot_d, \bigotimes_{j=1}^{t} \bS^{\lambda^j} M_j^{[d]}\right)\cong \bigotimes_{j=1}^{t}\bS^{\lambda^j}H^\bullet(V\otimes M_j).
\]
\end{itemize}
\end{theorem}
\begin{rem}
    The universal quotient $p^*V \to \cQ$ induces natural maps $H^{\bullet}(V \otimes M_j) \to H^{\bullet}(M_j^{[d]})$ for each $j$. The proof below will additionally show that the isomorphism in part (b) is given by taking Schur functors and tensoring these natural maps.
\end{rem}
\begin{proof}[Proof of Theorem \ref{thm:multiple_sub_quot}]
    Fix sufficiently large integer $m$ such that Theorem~\ref{thm:two_insertions}(ii) and Lemma~\ref{lem:two_L^{m}} hold for Schur functors of consecutive line bundles $L_{m-1}=\cO(m-1)$ and $L_{m}=\cO(m)$. The key idea here is to represent each $K_i^{\{d\}}$ and $M_j^{[d]}$ as two-term complexes involving only $L_{m-1}^{[d]}$ and $L_m^{[d]}$, so that we may analyze Schur functors of the latter instead.
    
    Note that $R^1\pi_*(\cF\otimes p^{*}L_{m-1}) = 0$, where $\pi$ and $p$ are the projections maps \eqref{eq:universalSeq} for $\quot_d\times \bP^1$. The diagonal sequence for $\bP^1$ gives an exact sequence  \cite[Proposition 1.1]{Stromme}
      \[
    \begin{tikzcd}
    0 \arrow[r] & \pi^*L_{m-1}^{[d]}\otimes p^*\cO(-1) \arrow[r]         & \pi^*L_{m}^{[d]} \arrow[r]         & \cF\otimes p^{*}L_{m} \arrow[r]         & 0
    \end{tikzcd}
    \]
    on $\quot_d\times \bP^1$. For any line bundle $M$, twist the above sequence by $p^*(M\otimes L_m^{\vee} )$, and push forward via $\pi$ to obtain an exact triangle 
   \begin{equation}\label{eq:M^[d]_resol_Lm}
       \begin{tikzcd}
 L_{m-1}^{[d]} \otimes R\pi_*p^*(M\otimes L_{m-1}^\vee) \arrow[r, "\Phi_{M}"]         & {L_m^{[d]} \otimes R\pi_*p^*(M\otimes L_{m}^{\vee})})   \arrow[r] &      M^{[d]} \arrow[r,"+1"] & \
    \end{tikzcd}
   \end{equation}
    in the derived category of coherent sheaves on $\quot_d$. Explicitly, $R\pi_*p^*(M\otimes L_{m}^{\vee}) = H^\bullet(M\otimes L_m^\vee) \otimes \cO_{\quot_d}$
    is a trivial vector bundle supported in degree $0$ or $1$; the complex is identically zero if $\deg M \otimes L_m^\vee = -1$. This exhibits $M^{[d]}$ as quasi-isomorphic to a two-term complex of vector bundles, denoted $\Phi_M$, which is
    \[
    \Phi_M = \begin{cases}
        L_{m-1}^{[d]} \otimes H^0(M \otimes L_{m-1}^{\vee}) \to 
        L_{m}^{[d]} \otimes H^0(M \otimes L_{m}^{\vee}) & \text{if }
        \deg(M\otimes L_m^{\vee}) \geq 0, \\ 
        L_{m-1}^{[d]} \otimes H^1(M \otimes L_{m-1}^{\vee}) \to 
        L_{m}^{[d]} \otimes H^1(M \otimes L_{m}^{\vee}) & \text{if } \deg(M\otimes L_m^{\vee}) \leq -1,
    \end{cases}
    \]
    where the first complex $(\Phi_M)_{\bullet}$ is viewed as concentrated in degrees $[-1,0]$, and the second $(\Phi_M)^{\bullet}$ is concentrated in degrees $[0,1]$.
    Similarly, replacing $\cF$ by $\cE$ in the above argument, we obtain the exact triangle
    \begin{equation}\label{eq:res_K_L_m_sub}
          \begin{tikzcd}
 L_{m-1}^{\{d\}} \otimes R\pi_*p^*(K\otimes L_{m-1}^\vee) \arrow[r, "\Psi_K"]         & {L_m^{\{d\}} \otimes R\pi_*p^*(K\otimes L_{m}^{\vee})})   \arrow[r] &      K^{\{d\}} \arrow[r,"+1"] & \
    \end{tikzcd}
     \end{equation}
     for any line bundle $K$.
     
     Thus, our problem is reduced to computing the hypercohomology groups of the complex 
     \[
        G^\bullet:= \bigotimes_{i=1}^{s} \bS^{\mu^i}(\Psi_{K_i})^\bullet \otimes \bigotimes_{j=1}^{t}\bS^{\lambda^j} (\Phi_{M_j})^\bullet.
     \]
    Note that we write $\bS^{\mu^i} (\Psi_{K_i})^\bullet$, $\bS^{\lambda^j} (\Phi_{M_j})^\bullet$ instead of $\bS^{\mu^i} (\Psi_{K_i})_\bullet$, $\bS^{\lambda^j} (\Phi_{M_j})_\bullet$ for all $i, j$ for ease of presentation, but apply the corresponding construction in Definition \ref{def:Schur_complex} depending on whether $\Psi_{K_i}$, $\Phi_{M_j}$ are concentrated in degrees $[0,1]$ or $[-1,0]$. Now the degree $q$ term $G^q$ in the complex $G^{\bullet}$ splits as direct sum of vector bundles of the form
     \[
     \bS^{\alpha}L_{m-1}^{\{d\}}\otimes \bS^{\beta}L_{m}^{\{d\}} \otimes \bS^{\gamma}L_{m-1}^{[d]}\otimes \bS^{\delta}L_{m}^{[d]}
     \]
     for some partitions $\alpha,\beta,\gamma,\delta$ satisfying $$|\alpha|+|\beta|=\lvert \mu^1\rvert+\cdots+|\mu^s|\quad\text{and}\quad |\gamma|+|\delta|=|\lambda^1|+\cdots+|\lambda^t|. $$ If $\mu^i$ is nontrivial for some $i$, then either $\alpha$ or $\beta$ is nontrivial, and thus, all its cohomology groups vanish by Lemma~\ref{lem:two_L^{m}}. Therefore, the first page of the hypercohomology spectral sequence for $G^\bullet$ is identically zero, so we get $H^D(\quot_d,G^\bullet)=0$ for all $D\ge 0$. This concludes the proof of part (a).
     
     To prove part (b), we will use the exact triangle 
   \begin{equation}
   \label{eqn: m_j_sub_quot}
    \begin{tikzcd}
M_{j}^{\{d\}} \arrow[r,"\Theta_{M_j}"] & R\pi_*(p^*(V\otimes M_j)) \arrow[r,] & M_j^{[d]} \arrow[r,"+1"] & \
    \end{tikzcd}   
   \end{equation}
where $R\pi_*(p^*(V\otimes M_j))\cong H^{\bullet}(V\otimes M_j)\otimes \cO_{\quot_d}$. Let us first observe that, when $\deg M_j \ge d+b$, it is a vector bundle supported in degree zero, and whenever $\deg M_j < 0$, the complex $M_j^{\{d\}}$ is vector bundle $R^1\pi_*(p^*M_j\otimes \cS )$ supported in degree one. 

When $\deg M_j\ge d+b$, we  apply Proposition~\ref{prop:SchurExact} to see that $\bS^{\lambda^j}M_j^{[d]}$ is quasi-isomorphic to $\bS^{\lambda^j} (\Theta_{M_j})_\bullet$, supported in degrees $[-|\lambda^j|,0] $, with $(-q)$-th term given by
\[
\bS^{\lambda^j} (\Theta_{M_j})_q= \bigoplus_{\alpha\subset\lambda^j,|\alpha|=s}\bS^{\alpha^\dagger} M_j^{\{d\}}\otimes\bS^{\lambda^j/\alpha}\left(H^0(V\otimes M_j)\otimes \cO_{\quot_d} \right)
\]
On the other hand, when $\deg M_j<0$, we see that $\bS^{\lambda^j} M_j^{[d]}$ is quasi-isomorphic to $\bS^{\lambda^j} (\Theta_{M_j})^\bullet$, supported in degrees $[0,|\lambda^j|]$, with $q$-th term given by
\[
\bS^{\lambda^j} (\Theta_{M_j})^q= \bigoplus_{\alpha\subset\lambda^j,|\alpha|=s}\bS^{\lambda^j/\alpha} (M_j^{\{d\}}[1])\otimes\bS^{\alpha^\dagger}\left(H^1(V\otimes M_j)\otimes \cO_{\quot_d} \right),
\]
where $M_j^{\{d\}}[1]=R^1\pi_*(p^*M_j\otimes \cS)$.
Now consider the complex
\[
E^\bullet=\bigotimes_{j=1}^{t}\bS^{\lambda^j} (\Theta_{M_j})^\bullet,
\]
where  we use cohomology notation for ease of presentation again. Observe that all cohomology groups of each summand of each term $E^q$ vanish by part(a), except for the $D$-th term, with $D=\sum_{\deg M_j < 0}|\lambda^j|$, which contains the summand 
\begin{equation}\label{eq:D-th_term_complex}
    \bigotimes_{ \deg M_j < 0} \bS^{(\lambda^j)^{\dagger}} \left(H^1(V\otimes M_j)\otimes \cO_{\quot_d}\right)
        \otimes 
        \bigotimes_{\deg M_j \geq 0 } \bS^{\lambda^j} \left(H^0(V \otimes M_j)\otimes \cO_{\quot_d}\right).
\end{equation}
Therefore, the hypercohomology spectral sequence degenerates on page two, and gives us
\[
H^D(\quot_d,E^\bullet) = \bigotimes_{ \deg M_j < 0} \bS^{(\lambda^j)^{\dagger}} H^1(V\otimes M_j)
        \otimes 
        \bigotimes_{\deg M_j \geq 0 } \bS^{\lambda^j} H^0(V \otimes M_j).
\]
for $D=\sum_{\deg M_j < 0}|\lambda^j|$, while all other cohomology groups vanish. Note that in computing $H^D$ from \eqref{eq:D-th_term_complex}, we use the fact that $\cO_{\quot_d}$ has no higher cohomology groups, which is a corollary of Theorem \ref{thm:two_insertions}, as  explained in Remark \ref{rem:structure_sheaf}.

Now assume $0\le \deg M_j\le d+b-1\le m-1$. In this case $M_j^{\{d\}}$ is supported in degrees $[0,1]$, and is quasi-isomorphic to $ (\Psi_{M_j})^\bullet$ using the triangle~\eqref{eq:res_K_L_m_sub}. The latter complex is supported in degrees $[0,1]$ and is given by 
\begin{equation*}
          \begin{tikzcd}
 L_{m-1}^{\{d\}} \otimes H^1(M_j\otimes L_{m-1}^\vee) \arrow[r, "\Psi_{M_j}"]         & {L_m^{\{d\}} \otimes H^1(M_j\otimes L_{m}^{\vee})}) .
    \end{tikzcd}
     \end{equation*}
Thus, using the triangle~\eqref{eqn: m_j_sub_quot}, we get a quasi-isomorphism of $M_j^{[d]}$ with the complex of vector bundle concentrated in degrees $[-1,0]$ given by 
     \[
          \begin{tikzcd}
 L_{m-1}^{\{d\}} \otimes H^1(M_j\otimes L_{m-1}^\vee) \arrow[r,]         & {L_m^{\{d\}} \otimes H^1(M_j\otimes L_{m}^{\vee})})\oplus H^{0}(V\otimes M_j)\otimes \cO_{\quot_d},
    \end{tikzcd}
     \]
     and argue similarly.
\end{proof}

\subsection{Vanishing without size constraints}
We now turn to the setting when there is no restrictions on the sizes of partitions. We will extend the result for two consecutive line bundles in Proposition~\ref{prop:cohomological_degree_no_size_constraints} to showing vanishing of higher cohomology groups for multiple line bundle.
\begin{theorem} \label{prop: onlyH0-intro}
     Assume $V$ splits as a direct sum of line bundles of nonpositive degrees.
     Let $M_1, M_2, \ldots, M_t$ be line bundles on $\bP^1$ of degrees at least $ d+b$. For any tuple of partitions $\lambda^1, \lambda^2, \ldots, \lambda^t$,  we have
    \[
     H^i \left(\Quot,\bigotimes_{j = 1}^t \bS^{\lambda^j}M_j^{[d]}\right) = 0\quad \text{for all } i\ge1. 
    \]
\end{theorem}
\begin{proof}
    Let $m=d+b$, and thus Proposition~\ref{prop:cohomological_degree_no_size_constraints} applies whenever two consective line bundles are involved. Let $M_1,M_2,\dots,M_{t}$ be line bundles with degrees at least $m$, and consider the short exact sequences \eqref{eq:M^[d]_resol_Lm} of vector bundles ,
  \begin{equation*}
       \begin{tikzcd}
 0\arrow[r]& L_{m-1}^{[d]} \otimes H^0(M_j\otimes L_{m-1}^\vee) \arrow[r, "\Phi_{M_j}"]         & {L_m^{[d]}\otimes  H^0(M_j\otimes L_{m}^{\vee})})   \arrow[r] &      M_j^{[d]} \arrow[r,] & 0
    \end{tikzcd}
   \end{equation*}
   for each $1\le j\le t$. For any partitions $\lambda^1,\lambda^2,\dots,\lambda^t$, we consider the complex of vector bundles
   \[
   G^\bullet=\bigotimes_{j=1}^{t}\bS^{\lambda^j} (\Phi_{M_j})_\bullet
   \]
   Note that the complex is supported in degrees $[-\ell,0]$  where $\ell=(|\lambda^1|+\cdots+|\lambda^t|)$. Proposition~\ref{prop:SchurExact} implies that the complex 
   \[
   0\to G^{-\ell}\to G^{-\ell+1}\to \cdots\to G_0\to \bigotimes_{j=1}^{t}\bS^{\lambda^j} {M_j}
   \]
   is exact. The $(-q)$-th term $G^{-q}$ splits as a direct sum with each summand of the form
   \[
   \bS^{\gamma}L_{m-1}^{[d]}\otimes \bS^{\delta}L_{m}^{[d]},
   \]
   which by Proposition~\ref{prop:cohomological_degree_no_size_constraints} can only have zeroth cohomology. In particular, $H^i(\quot_d,G^{-q})=0$ for all $0\le q\le \ell$ and $i\ge 1$, so we conclude that $$H^{i}\left(\quot_d,\bigotimes_{j=1}^{t}\bS^{\lambda^j} {M_j}\right)=0\qquad\text{for all } i\ge 1.$$
   \end{proof}

 We  now provide some concrete examples explaining why the size constraint is necessary in the statement of Theorem~\ref{thm:intro_vb_insertion}, even though all higher cohomology groups vanish for positive enough line bundles.
\begin{example}\label{ex:size_constraints}
    Let $V = \cO_{\bP^1}^{\oplus 2}$, $r = 1$ and $d = 2$. For the line bundle $L_4 =\cO_{\bP^1}(4)$, the natural map \[\wedge^6 H^0(V\otimes L_4) \to H^0(\quot_d, \wedge^6 L_4^{[2]})\] is not an isomorphism. This shows that the bound $|\lambda| < (dn+n)/(n-r)$ is sharp in Theorem~\ref{thm:multiple_insertion_intro}.
    
    Our computation uses the St\o mme embedding of $\iota: \quot_d \hookrightarrow \Gr(3,10) \times \Gr(4,12)$, corresponding to the twists $m = 4, 5$. 
    We have $\wedge^6 L_4^{[2]} = \iota^* \bS^{(1^6)}\cB_1$, which yields a resolution of $\iota_* \wedge^6 L_4^{[2]}$ as in \eqref{eq:resolution_single_taut}. This then gives rise to a spectral sequence $H^k(\cV_t) \Rightarrow H^{k-t}(\quot_d, \wedge^6 L_4^{[2]})$. By using the computer to run an exhaustive search, we find that there are exactly two nonzero terms on the initial page of the spectral sequence: in addition to the term $H^0(\cV_0) = \wedge^6 H^0(V\otimes L_4)$, which has dimension $\dim H^0(\cV_0) = 210$, the cohomology group $H^{23}(\cV_{24})$ is also nonzero. 
    
    Specifically, let $W_1 = H^0(V \otimes L_4), W_2 = H^0(V\otimes L_5)$, and let $\mu = (10,10,4), 
        \sigma = (6, 6, 2^6)$.
    Then the multiplicity 
    \[
    \sum_{\alpha, \beta} c^{\mu^{\dagger}}_{\alpha, \beta} c^{\sigma}_{\alpha, \beta} = 28
    \]
    is nonzero, and
    \begin{align*}
        H^{23}(\Gr(3,10)\times \Gr(4,12), \cV_{24}) &= 
        H^{23}(
         (\bS^{\mu}\cA_1 \otimes \bS^{(1^6)} \cB_1  ) 
        \boxtimes \bS^{\sigma}\cB_2^{\vee})^{\oplus 28} \\ 
        &= (\bS^{(3^{10})}W_1 \otimes \bS^{((2)^{12})}{W_2^{\vee}})^{\oplus 28} \\ 
        & \cong \left ( (\det W_1)^{\otimes 3} \otimes (\det W_2^{\vee})^{\otimes 2} \right )^{\oplus 28}.
    \end{align*}
    One possible choice of $\alpha, \beta$ would be
    \begin{align*}
        \alpha &= (3,3), \\ 
        \beta &= (3,3,2^6),
    \end{align*}
    for which the multiplicity is $c^{\mu^{\dagger}}_{\alpha, \beta} c^{\sigma}_{\alpha, \beta} = 1$. 
    
    Since there are only two nonzero terms on the initial page and $\wedge^6 L_4^{[2]}$ cannot have cohomology in a negative degree, the induced map $H^{23}(\cV_{24}) \to H^0(\cV_0)$ must be injective. In fact, there is a short exact sequence
    \[
    0 \to \left ( (\det W_1)^{\otimes 3} \otimes (\det W_2^{\vee})^{\otimes 2} \right )^{\oplus 28} \to \wedge^6 W_1 \to H^0(\quot_d, \wedge^6 L_4^{[2]}) \to 0.
    \]
\end{example}
\begin{example}
    Let $V = \cO_{\bP^1}^{\oplus 3}, r = 1, d = 3$. Then \[H^2(\quot_d, \Sym^2 (L_2^{[d]})^{\vee}) = H^2(\quot_d, \bS^{(2)} (L_2^{[d]})^{\vee}) \neq 0.\] Note that the partition $\nu = (2)$ satisfies the constraint that $|\nu| < (nd+n)/r = 12$, but fails the condition of part~(i) of Theorem~\ref{thm:Ext_implications} because $\nu_1 = n-r$ is too large. This example also shows that the same theorem cannot be upgraded to the statement that for $|\nu| < (nd+n)/r$, 
    \[
    H^{\geq \nu_1 - (n-r-1)}(\Quot,  
    (\bS^{\nu}L_{m}^{[d]})^{\vee}) = 0.
    \]
    To see the non-vanishing of $H^2$, we consider the Str\o mme embedding $\iota: \quot_d \hookrightarrow \Gr(3,9) \times \Gr(5,12)$, corresponding to the twists $m = 2, 3$. Then we run an exhaustive search of all vector bundles 
    \[
        \cV^{\mu} = \bS^{\mu} \cA_1 \otimes \Sym^2 \cB_1^{\vee} \boxtimes \bS^{\mu^{\dagger}}(\cB_2^{\vee} \oplus \cB_2^{\vee})
    \]
    on $\Gr(3,9) \times \Gr(5,12)$, where $|\mu| \leq \codim(\Quot \subseteq \Gr(3,9) \times \Gr(5,12)) = 42$. Let $\mu_1 = (8,2,2)$ and $\mu_2 = (7,2,2)$, and let $W_1 = H^0(V \otimes L_2)$, $W_2 = H^0(V \otimes L_3)$. We find that the only vector bundles with nonzero cohomology groups arise from these two partitions:
    \begin{align*}
        H^{13}(\cV^{\mu_1}) &= \bS^{((-1)^{8},-2)}(W_1^{\vee}) \otimes \bS^{(1^{12})}(W_2^{\vee})^{\oplus 7} = (W_1 \otimes \det W_1 \otimes \det W_2^{\vee})^{\oplus 7}, \\
        H^{13}(\cV^{\mu_2}) &= \bS^{((-1)^{9})}(W_1^{\vee}) \otimes \bS^{(1^{11})}(W_2^{\vee})^{\oplus 6} = (\det W_1 \otimes \wedge^{11} W_2^{\vee})^{\oplus 6}.
    \end{align*}
    The spectral sequence computing the cohomology of $\Sym^2 (L_2^{[d]})^{\vee}$ in this case reduces to an exact sequence 
    \[
    0 \to H^1(\Sym^2 (L_2^{[d]})^{\vee}) \to 
    (W_1 \otimes \det W_1 \otimes \det W_2^{\vee})^{\oplus 7} 
    \to (\det W_1 \otimes \wedge^{11} W_2^{\vee})^{\oplus 6} 
    \to
     H^2(\Sym^2 (L_2^{[d]})^{\vee}) \to 
     0.
    \]
    While we don't know how to describe this middle map explicitly, the cokernel $H^2(\Sym^2 (L_2^{[d]})^{\vee})$ must be nonzero for dimension reasons. 
\end{example}

\bibliographystyle{alphnum}
\bibliography{ref}

\end{document}